\documentclass[10pt,english]{article}
\usepackage{amssymb}
\usepackage[english]{babel}
\usepackage{graphicx} 
\usepackage{psfrag}
\usepackage[utf8]{inputenc}
\usepackage[T1]{fontenc}
\usepackage{babel}
\usepackage{amsmath}
\usepackage{verbatim}
\usepackage{amsthm}
\usepackage{array}
\usepackage{times}
\usepackage{float}
\usepackage[hang,sc,small]{caption}

\theoremstyle{definition}

\newtheorem{theorem}{Theorem}
\newtheorem{lemma}{Lemma}
\newtheorem{remark}{Remark}

\newtheorem{proposition}[theorem]{Proposition}

\newcommand{\R}{\mathbb{R}}

\DeclareMathOperator*{\minimize}{minimize}
\DeclareMathOperator*{\maximize}{maximize}
\DeclareMathOperator*{\subto}{subject\; to}

\begin{document}

\title{Algorithms for the continuous nonlinear resource allocation
  problem---new implementations and numerical studies} \author{Michael
  Patriksson\footnote{Department of Mathematical Sciences, Chalmers University
    of Technology, SE-412 96 Gothenburg, Sweden, e-mail: mipat@chalmers.se}
  \/ and Christoffer Str\"{o}mberg\footnote{\"Ostra Allekullav\"{a}gen 98, SE-43
    991 Onsala, Sweden, e-mail: chris.stromberg@icloud.com}} \maketitle

\begin{abstract}
Patriksson~\cite{Pat08} provided a then up-to-date survey on the continuous,
separable, differentiable and convex resource allocation problem with a single
resource constraint. Since the publication of that paper the interest in the
problem has grown: several new applications have arisen where the problem at
hand constitutes a subproblem, and several new algorithms have been developed
for its efficient solution. This paper therefore serves three purposes. First,
it provides an up-to-date extension of the survey of the literature of the
field, complementing the survey in Patriksson~\cite{Pat08} with more then 20
books and articles. Second, it contributes improvements of some of these
algorithms, in particular with an improvement of the pegging (that is,
variable fixing) process in the relaxation algorithm, and an improved means to
evaluate subsolutions. Third, it numerically evaluates several relaxation
(primal) and breakpoint (dual) algorithms, incorporating a variety of pegging
strategies, as well as a quasi-Newton method. Our conclusion is that our
modification of the relaxation algorithm performs the best. At least for
problem sizes up to 30 million variables the practical time complexity for the
breakpoint and relaxation algorithms is linear.
%
\end{abstract}
\section{Introduction}\label{sec:intro}
We consider the continuous, separable, differentiable and convex resource
allocation problem with a single resource constraint. The problem is
formulated as follows: Let $J := \{ 1,2, \dotsc ,n \}$. Let $\phi_j:\mathbb{R}
\rightarrow \mathbb{R}$ and $g_j:\mathbb{R} \rightarrow \mathbb{R}$, $j \in
J$, be convex and continuously differentiable. Moreover, let $b \in
\mathbb{R}$ and $-\infty < l_j<u_j < \infty$, $j \in J$. Consider the problem
to
\begin{subequations}
\begin{align}
        \minimize_x &\quad \phi(\mathbf{x}) := \sum_{j \in J} \phi_j(x_j),
        \\ \subto & \quad g(\mathbf{x}):= \sum _{j \in J} g_j(x_j) \le
        b, \label{problemb} \\ & \quad l_j \le x_j \le u_j, \quad j \in
        J. \label{problemc}
\end{align}
\label{problem}
\end{subequations}

We also consider the problem where the inequality constraint (\ref{problemb})
is replaced by an equality, i.e.,
\begin{subequations}
\begin{align}
        \minimize_x &\quad \phi(\mathbf{x}) := \sum_{j \in J}
        \phi_j(x_j),\\ \subto & \quad g(\mathbf{x}):= \sum_{j \in J} a_jx_j =
        b, \label{problemeqb} \\ & \quad l_j \le x_j \le u_j, \quad j \in
        J, \label{problemeqc}
\end{align}
\label{problemeq}
\end{subequations}
where $a_j \neq 0$, $j \in J$, and the sign is the same for all $j \in
J$. Further, we assume that there exists an optimal solution to problems
(\ref{problem}) and (\ref{problemeq}). For brevity, in the following
discussions we define $X_j := [l_j, u_j]$, $j \in J$.

Problems (\ref{problem}) and (\ref{problemeq}) arise in many areas, e.g., in
search theory (\cite{Koo99}), economics (\cite{Mar52}), stratified sampling
(\cite{BRS99}), inventory systems (\cite{MaK93}), and queuing manufacturing
networks (\cite{BiT89}). Further, these problems occur as subproblems in
algorithms that solve the integer resource allocation problem
(\cite[Section~4.7]{Mje83}, \cite[pp. 72--75]{IbK88}, and \cite{BrS02}),
multicommodity network flows (\cite[Section~4.2]{Sho85}), and several
others. Moreover, problems (\ref{problem}) and (\ref{problemeq}) can be used
as subproblems when solving resource allocation problems with more than one
resource constraint (\cite{Mje83,FeZ83}), and to solve extensions of problems
(\ref{problem}) and (\ref{problemeq}) to a nonseparable objective function
$\phi$ (\cite{Mje83,DSV07}); The books \cite{Mje83,IbK88,Lus12} describe
several extensions, such as to minmax/maxmin objectives, multiple time
periods, substitutable resources, network constraints, and integer decision
variables.

Many numerical studies of the problems (\ref{problem}) and (\ref{problemeq})
have been performed; for example, see
\cite{BiH81,NiZ92,RJL92,KoL98,Kiw07,Kiw08a,Kiw08b}. Our numerical study is
however timely and well motivated, since except for those by Kiwiel
\cite{Kiw07,Kiw08a,Kiw08b}, where the quadratic knapsack problem is studied,
none of the earlier approaches study large-scale versions of problem
(\ref{problem}) or (\ref{problemeq}). There are also several algorithms (e.g.,
\cite[Section 1.4]{NiZ92} and \cite{Ste01}) which are claimed to be promising,
but have not been evaluated in a large-scale study. Only one earlier study
(\cite{KoL98}) evaluates the performance of algorithms for the problems
(\ref{problem}) and (\ref{problemeq}) with respect to variations in the
portions of the variables whose values are at a lower or upper bound at the
optimal solution (see Section \ref{sec:dop}), and this is done for modest size
instances ($n = 10^4$) only. Further, no study has been done on the
computational complexity for non-quadratic versions of the problems
(\ref{problem}) or (\ref{problemeq}). Our numerical study also incorporates
improvements of the relaxation algorithm, as presented in
Sections~\ref{sec:bh4}--\ref{sec:tandf}, and utilizes performance profiles
(\cite{DoM02}).

As a final note on the computational tests, we only consider problem instances
where the dual variable corresponding to the resource constraint
(\ref{problemb}), respectively (\ref{problemeqb}), can be found in closed
form; otherwise, we would need to implement a numerical method in some of the
steps, e.g., a Newton method. We consider only customized algorithms for the
problem at hand, since we presume that they perform better than more general
algorithms under the above assumption.

Patriksson \cite{Pat08} presents a survey of the history and applications of
problems (\ref{problem}) and (\ref{problemeq}). Since its publication several
related articles have been published; the survey of \cite{Pat08} is therefore
complemented in Section~\ref{sec:extsur}. Section~\ref{sec:rank} presents a
framework of breakpoint algorithms, resulting in three concrete
representatives.  Section~\ref{sec:relaxation} presents a framework of
relaxation algorithms, and ultimately six concrete example methods. In Section
\ref{sec:zen_algor} we describe a quasi-Newton method, due to Nielsen and
Zenios~\cite{NiZ92}, for solving the problem
(\ref{problemeq}). Section~\ref{sec:method} describes the numerical study. A
test problem set is specified and the performance profile used for the
evaluation is defined.  In Section~\ref{sec:compexp}, we analyze the results
from the numerical study. The structure is such that we first compare the
relaxation algorithms, second the pegging process, and third the best
performing algorithms among these two with the quasi-Newton method. Finally,
we draw overall conclusions.

\section{Extension of the survey in \cite{Pat08} } \label{sec:extsur}

We here extend the survey in \cite{Pat08}, using the same taxonomy, and sorted
according to publication date.

\begin{description}
\item[{\cite{Mje83}}] {\sc K.\ M.\ Mjelde}, {\em Methods of the Allocation of
  Limited Resources}, Section 4.7
\begin{description}
\item[{\sf (Problem)}] $\phi_j \in C^2$; linear equality ($a_j=1$); $l_j=0$
\item[{\sf (Methodology)}] The ranking algorithm of \cite{LuG75}
\item[{\sf (Citations)}] Applications in capital budgeting
  (\cite{Han68,Shi77}), cost-effectiveness problems
  (\cite{Kir68,Pac70,Mje78}), health care (\cite{Fet73}), marketing
  (\cite{LuG75}), multiobjective optimization (\cite{Geo67}), portfolio
  selection (\cite{JuD75}), production (the internal report leading to
  \cite{BiH79}), reliability (\cite{Bod69}), route-planning for ships or
  aircraft (\cite{DBR66}), search (\cite{ChC58}), ship loading (\cite{Kyd69}),
  and weapons selection (\cite{Dan67})
\item[{\sf (Notes)}] A monograph on resource allocation problems containing a
  comprehensive overview of the resource allocation problem, including
  extensions to several resources, non-convex or non-differentiable
  objectives, integral decision variables, fractional programming
  formulations, etcetera.
\end{description}
\end{description}

\begin{description}
\item[{\cite{Sho85}}] {\sc N.\ Z.\ Shor}, {\em Minimization Methods for
  Nondifferentiable Functions}, Section 4.2
\begin{description}
\item[{\sf (Problem)}] $\phi_j(x_j) = \frac{1}{2}(x_j-y_j)^2$; linear equality
  ($a_j=1$); $l_j=0$
\item[{\sf (Methodology)}] Pegging
\item[{\sf (Citations)}] \cite{ShI69}, in which the motivating linear
  programming application is described
\item[{\sf (Notes)}] The problem arises within the framework of a right-hand
  side allocation algorithm for a large-scale linear program.
\end{description}
\end{description}

\begin{description}
\item[{\cite{HuZ05}}] {\sc Z.-S.\ Hua and B.\ Zhang}, {\em Direct algorithm for
  separable continuous convex quadratic knapsack problem} (in Chinese)
\begin{description}
\item[{\sf (Problem)}] $\phi_j(x_j) = \frac{q_j}{2}x^2_j - r_jx_j$; linear
  inequality ($a_j > 0$); $l_j=0$
\item[{\sf (Methodology)}] Pegging
\item[{\sf (Citations)}] Algorithms for the problem
  (\cite{PaK90,MeR00,BrS02,BrS02b}) as well as for the case of integer
  variables
\item[{\sf (Notes)}] A numerical illustration ($n = 6$).
\end{description}
\end{description}

\begin{description}
\item[{\cite{DaF06}}] {\sc Y.-H.\ Dai and R.\ Fletcher}, {\em New algorithms
  for singly linearly constrained quadratic programs subject to lower and
  upper bounds}
\begin{description}
\item[{\sf (Problem)}] $\phi_j(x_j) = \frac{q_j}{2}x^2_j - r_jx_j$, $q_j>0$;
  $g_j$ convex in $C^2$ with $g'(x_j) > 0$
\item[{\sf (Methodology)}] A combination of a bracketing algorithm on the
  Lagrangian dual derivative, and a secant algorithm for the Lagrangian dual
  problem
\item[{\sf (Citations)}] Algorithms for the problem
  (\cite{HKL80,Bru84,CaM87,PaK90})
\item[{\sf (Notes)}] The problem arises as a subproblem in a gradient
  projection method for a general quadratic programming problem over a scaled
  simplex.
\end{description}
\end{description}

\begin{description}
\item[{\cite{LiS06}}] {\sc D.\ Li and X.\ Sun}, {\em Nonlinear Integer
  Programming}, Chapter 6: {\rm Nonlinear Knapsack Problems}, Section 6.1:
  {\rm Continuous--Relaxation-based Branch--and--Bound Methods}
\begin{description}
\item[{\sf (Problem)}] $\phi_j$ and $g_j$ increasing; $g_j$ convex in $C^2$
  with $g'_j > 0$
\item[{\sf (Methodology)}] Multiplier search
\item[{\sf (Citations)}] Multiplier search methods (\cite{BrS95}), pegging
  methods (\cite{BrS02,BrS02b})
\item[{\sf (Notes)}] The problem arises as a subproblem in branch--and--bound
  methods for the integer programming version of the problem, such as for the
  quadratic knapsack problem, stratified sampling, manufacturing capacity
  planning, linearly constrained redundancy optimization in reliability
  networks, and linear cost minimization in reliability networks.
\end{description}
\end{description}

\begin{description}
\item[{\cite{DSV07}}] {\sc K.\ Dahiya, S.\ K.\ Suneja and V.\ Verma}, {\em
  Convex programming with a single separable constraint and bounded variables}
\begin{description}
\item[{\sf (Problem)}] $\phi_j(x_j) = \frac{q_j}{2}x^2_j - r_jx_j$, $q_j>0$;
  $g_j$ convex in $C^2$ with $g'(x_j) > 0$; studies also the special case of a
  linear equality
\item[{\sf (Methodology)}] Iterative descent process using strictly convex
  quadratic separable approximations of a nonseparable original objective $f
  \in C^2$; subproblems solved using the pegging algorithm of \cite{Ste01}
\item[{\sf (Citations)}] General references on convex programming over box
  constraints; \cite{HKL80,DFL86,PaK90} for example algorithms for separable
  convex programming
\item[{\sf (Notes)}] Numerical QP ($n=6$, $g_j$ quadratic) illustration;
  numerical comparison with an augmented Lagrangian algorithm for a small
  problem ($n=2$). 
\end{description}
\end{description}

\begin{description}
\item[{\cite{Kiw07}}] {\sc K.\ C.\ Kiwiel}, {\em On linear-time algorithms for
  the continuous quadratic knapsack problem}
\begin{description}
\item[{\sf (Problem)}] $\phi_j(x_j) = \frac{q_j}{2}x^2_j - r_jx_j$, $q_j>0$;
  linear equality
\item[{\sf (Methodology)}] Breakpoint search algorithm applying median search
  of all breakpoints
\item[{\sf (Citations)}] Breakpoint search algorithms:
  \cite{Bru84,CaM87,PaK90,MSMJ03}; sorting and searching methods:
  \cite{Knu98,Kiw05}
\item[{\sf (Notes)}] Develops a general $O(n)$ breakpoint algorithm;
  shows that the algorithms of \cite{PaK90,MSMJ03} may fail even on small
  examples; presents a modification of the breakpoint removal in the algorithm
  of \cite{CaM87}. Numerical experiments ($n \in [50 \cdot 10^3, 2 \cdot
    10^6]$) for uncorrelated, weakly, and strongly correlated data; the new
  algorithm wins in CPU time over those in \cite{Bru84} and \cite{CaM87} by
  23\%, and 21\%, respectively, on average.
\end{description}
\end{description}

\begin{description} 
\item[{\cite{Kiw08a}}] {\sc K.\ C.\ Kiwiel}, {\em Breakpoint searching
  algorithms for continuous quadratic knapsack problem}
\begin{description}
\item[{\sf (Problem)}] $\phi_j(x_j) = \frac{q_j}{2}x^2_j - r_jx_j$, $q_j>0$;
  linear equality
\item[{\sf (Methodology)}] A family of breakpoint search algorithms that
  include several choices of breakpoints for a median search and updates of
  quantities used for evaluating the piecewise linear implicit constraint
  function at the median point
\item[{\sf (Citations)}] Applications in resource allocation
  (\cite{BiH81,BrS97,HoH95}), hierarchical production planning (\cite{BiH81}),
  network flows and transportation (\cite{HKL80,ShM90,Ven91,NiZ92,CoH94}),
  constrained matrix problems (\cite{CDZ86}), quadratic integer programming
  (\cite{BSS95,BSS96,HoH95}), Lagrangian relaxation (\cite{HWC74}), and
  quasi-Newton methods (\cite{CaM87}); $O(n \log n)$ sorting algorithms for
  the solution of the Lagrangian dual problem (\cite{HWC74,HKL80}), $O(n)$
  algorithms based on median search
  (\cite{Bru84,CaM87,MaD89,PaK90,CoH94,HoH95,MMP97}) and approximate median
  search methods with $O(n)$ average-case performance (\cite{PaK90}); primal
  pegging algorithms with $O(n^2)$ worst-case performance
  (\cite{Zip80,BiH81,Mic86,Ven91,RJL92,BSS96})
\item[{\sf (Notes)}] Develops several variants of $O(n)$ breakpoint search
  algorithms, including some ideas earlier proposed in, e.g.,
  \cite{PaK90,CoH94,HoH95,MMP97}; remarks that the more complex choices made
  in \cite{MaD89,PaK90,CoH94,HoH95,MMP97} also means that for some simple
  cases cycling may occur, and also provides convergent modifications for each
  of them. Numerical experiments ($n \in [50 \cdot 10^3, 2 \cdot 10^6]$) for
  uncorrelated and weakly and strongly correlated data, and for both exact and
  inexact computations of the median; comparisons made with the $O(n)$
  versions from \cite{Bru84,CaM87}, reporting that a version (Algorithm 3.1)
  using exact medians is about 20\% faster than the other ones; refers to an
  as yet unavailable technical report from 2006 for more extensive tests and
  comparisons with pegging methods.
\end{description}
\end{description}

\begin{description} 
\item[{\cite{Kiw08b}}] {\sc K.\ C.\ Kiwiel}, {\em Variable fixing algorithms
  for continuous quadratic knapsack problem}
\begin{description}
\item[{\sf (Problem)}] $\phi_j(x_j) = \frac{q_j}{2}x^2_j - r_jx_j$, $q_j>0$;
  linear equality
\item[{\sf (Methodology)}] Pegging
\item[{\sf (Citations)}] Applications: same references as in \cite{Kiw08a};
  breakpoint search methods
  (\cite{HWC74,HKL80,CaM87,MaD89,PaK90,CoH94,HoH95,MMP97,MSMJ03,Kiw07};
  pegging methods (\cite{LuG75,BiH79,BiH81,Sho85,Mic86,Ven91,RJL92,BSS96}) 
\item[{\sf (Notes)}] Develops a basic pegging algorithm and proposes several
  implementational choices for the solution of the reduced problem and the
  updates; shows that the algorithms in \cite{Mic86,RJL92,BSS96} fail on a
  small example, and that there is a gap in the convergence analysis in
  \cite{BiH81} (which also is closed) that affects algorithms whose analyses
  rest on that in \cite{BiH81} (e.g., \cite{Ven91}); provide more efficient
  versions of several of these methods, including the introduction of
  incremental updates which reduce computations, and a more efficient stopping
  criterion. Numerical experiments ($n \in [50 \cdot 10^3, 2 \cdot 10^6]$) for
  uncorrelated and weakly and strongly correlated data; comparisons made with
  the breakpoint search method of \cite{Kiw08a} which uses exact medians; on
  average the latter is 14\% slower while at the same time it has a more
  stable run time; it is remarked that the advantage of pegging over
  breakpoint search has been reported also in \cite{Ven91,RJL92}.
\end{description}
\end{description}

\begin{description} 
\item[{\cite{ZhH08}}] {\sc B.\ Zhang and Z.\ Hua}, {\em A unified method for a
  class of convex separable nonlinear knapsack problems}
\begin{description}
\item[{\sf (Problem)}] $\phi'_j/g'_j$ monotone and invertible, $g'_j$
  positive; consider $\sum_{j=1}^n g_j(x_j) \lhd \; b$ where $\lhd \in \{ \le,
  =, \ge \}$
\item[{\sf (Methodology)}] Pegging algorithm using binary search on the value
  of $\phi'_j/g'_j$
\item[{\sf (Citations)}] Applications: resource allocation
  (\cite{LuG75,Zip80,BiH81,Hoc94}), the singly constrained multi-product
  newsvendor problem (\cite{HaW63,Erl00,AMM04}), production/inventory problems
  (\cite{BSSW94,BSS95,BrS02}), stratified sampling (\cite{BSS95,BRS99,BrS02}),
  ``core subproblem'' (\cite{JuD75,AHKL80,MoT89,RJL92,BSS95,BSS96});
  algorithms: breakpoint methods (\cite{BrS95}, \cite{Ste01}, \cite{LuG75})
  and relaxation methods (\cite{KoL98}, \cite{BrS02}, \cite{BrS02b})
\item[{\sf (Notes)}] Claims (without a proof) that the complexity of the
  algorithm is $O(n)$, but the algorithm presented makes use of a mean-value
  method for the evaluation of the breakpoints which in the worst case results
  in $2n-1$ iterations. We also note that each iteration consist of $O(n)$
  operations, whence the complexity is $O(n^2)$. Numerical experiments ($n \in
  [10, 10^4]$) ($\phi$ quadratic, $g_j$ linear, $\lhd$ chosen uniformly
  randomly).
 \end{description}
\end{description}

\begin{description} 
\item[{\cite{WaW12}}] {\sc A.\ De Waegenaere and J.\ L.\ Wielhouwer}, {\em A
  breakpoint search approach for convex resource allocation problems with
  bounded variables}
\begin{description}
\item[{\sf (Problem)}] $\phi_j$ and $g_j$ convex; $g_j$ strictly monotone and
  $g_j([\phi_j'/g_j']^{-1})$ is either strictly increasing or strictly
  decreasing for all $j$
\item[{\sf (Methodology)}] Breakpoint search algorithm using a refined pegging
  method (5-sets pegging); generalizes the quadratic breakpoint algorithm in
  \cite{PaK90} and its extension in \cite{Kiw08a} such that it applies for the
  problem setting in \cite{BrS02}
\item[{\sf (Citations)}] Generalizes the quadratic breakpoint algorithm in
  \cite{PaK90} and its extension in \cite{Kiw08a} such that it is valid for
  $f_j$ and $g_j$ as in \cite{BrS02}. Applications in resource allocation
  \cite{PaK90,NiZ92,BrS95,VeW95,BSS96,BrS02,WaW02,CLZ03,BSSV06,Pat08}
\item[{\sf (Notes)}] Discuss their algorithms' advantages compared to other
  articles presented in \cite{Pat08}.
 \end{description}
\end{description}

\begin{description} 
\item[{\cite{KiW12}}] {\sc G.\ Kim and C.-H.\ Wu},
 {\em A pegging algorithm for separable continuous nonlinear knapsack problems
      with box constraints}
\begin{description}
\item[{\sf (Problem)}] $\phi'_j$ invertible: linear equality ($a_j > 0$)
\item[{\sf (Methodology)}] Pegging; improves the pegging algorithm in
  \cite{BiH81} by allowing the primal box constraints to be checked
  implicitly, using bounds on their Lagrange multipliers
\item[{\sf (Citations)}] Applications to portfolio selection \cite{Mar52},
  multicommodity network flows \cite{AHKL80}, transportation \cite{OhK84},
  production planning \cite{Tam80}
\item[{\sf (Notes)}] Compares their methodology with the method in
  \cite{BiH81} on random test problems. On randomized quadratic continuous
  knapsack problems ($n \in [5 \cdot 10^3, 2 \cdot 10^6]$) their algorithms
  wins by 8--10\%, except for the smallest problems where the method in
  \cite{BiH81} wins by $12.5\%$. Two other types of test problems are also
  investigated, wherein the quadratic continuous knapsack problem arises as a
  subproblem: quadratic network flows, and portfolio optimization. In the
  former case, the algorithm (based on conjugate gradients) is taken from
  \cite{Ara00}. Here, the pegging algorithm proposed wins over that in
  \cite{BiH81} by 10--48\%, with $n \in [200, 5 \cdot 10^3]$. In the
  portfolio optimization algorithm, which is based on a progressive hedging
  method described in \cite{Ara00}, the quadratic continuous knapsack problem
  arises as a subproblem. Here, the range of $n$ is not completely disclosed;
  however, the speed-up over the method in \cite{BiH81} is reported to be
  21--25\%.
\end{description}
\end{description}

\begin{description} 
\item[{\cite{Lus12}}] {\sc H.\ Luss}, {\em Equitable Resource Allocation:
  Models, Algorithms and Applications}, Chapter 2: {\em Nonlinear Resource
  Allocation}
\begin{description}
\item[{\sf (Problem)}] $\phi_j$ strictly convex; $g_j(x_j) = x_j$
 \item[{\sf (Methodology)}] Pegging
\item[{\sf (Citations)}] \cite{Koo53} as an origin; \cite{KoL98} for
  computational examples; \cite{Pat08} as a survey
\item[{\sf (Notes)}] This book extends the resource allocation problem
  (discussed only in Chapter 2) in several ways, including equitable
  optimization through the use of minmax/maxmin objectives, multiple time
  periods, substitutable resources, network constraints, and integer decision
  variables.
 \end{description}
\end{description}

\begin{description} 
\item[{\cite{ZhC12}}] {\sc B.\ Zhang and B.\ Chen}, {\em Heuristic and exact
  solution method for convex nonlinear knapsack problem}
\begin{description}
\item[{\sf (Problem)}] $\phi_j$ strictly convex; $g_j(x_j) = x_j$
 \item[{\sf (Methodology)}] The problem at hand is a subproblem in a
   branch--and--bound procedure for the solution of an integer-restricted
   version of the problem
\item[{\sf (Citations)}] Applications to the newsvendor problem
  (\cite{Erl00,AMM04}), resource allocation (\cite{BiH81,Hoc94}), production
  (\cite{BSS95}), and stratified sampling (\cite{BRS99}); efficient methods
  for the continuous relaxation (\cite{BrS95,KoL98,Ste01,ZhH08}); heuristics
  for the integer program based on rounding of the solution to the continuous
  relaxation (\cite{BrS95,HuZ06}); algorithms for the integer program based on
  the solution of continuous problems and branch--and--bound
  (\cite{BrS95,BrS02})
\item[{\sf (Notes)}] Utilizing the algorithm from \cite{ZhH08} to solve the
  continuous relaxations (and rounding to produce feasible solutions), the
  authors develop a branch--and--bound algorithm. It is compared with the
  methods from \cite{BrS95} as well as with branch--and--bound algorithms
  utilizing a variety of tree search principles, on instances with quadratic
  objectives, according to problem generation principles from \cite{BrS02} ($n
  \in \{10, 15, 500, 1000, 2000\}$).
 \end{description}
\end{description}

\begin{description} 
\item[{\cite{BGRS13}}] {\sc L.\ Bay{\'o}n, J.\ M.\ Grau, M.\ M.\ Ruiz, and
  P.\ M.\ Su{\'a}rez}, {\em An exact algorithm for the continuous quadratic
  knapsack problem via infimal convolution}
\begin{description}
\item[{\sf (Problem)}] $\phi_j$ strictly convex quadratic; linear equality
  ($a_j = 1$)
\item[{\sf (Methodology)}] Sorting of breakpoints
\item[{\sf (Citations)}] Previous algorithms for the problem
  (\cite{MeT93,CoH94})
\item[{\sf (Notes)}] Two numerical applications: (1) an economic dispatch
  problem ($n = 5$), and an academic example ($n \in [200, 10^4]$); in the
  latter example the results are favourably compared with the Matlab solver
  QUADPROG.
 \end{description}
\end{description}

\begin{description} 
\item[{\cite{DHH13}}] {\sc T.\ A.\ Davis, W.\ W.\ Hager, and
  J.\ T.\ Hungerford}, {\em The separable convex quadratic knapsasck problem}
\begin{description}
\item[{\sf (Problem)}] $\phi_j(x_j) = \frac{q_j}{2}x^2_j - r_jx_j$; $q_j>0$;
  linear equality
\item[{\sf (Methodology)}] Breakpoint search utilizing a heap data structure,
  based on an initial multiplier estimate using a secant Newton method
\item[{\sf (Citations)}] Applications (\cite{HKL80,CaM87,ShM90,NiZ92,DaF06});
  multiplier algorithms (\cite{HKL80,Bru84,CaM87,PaK90,MSMJ03}); pegging
  methods (\cite{BiH81,Sho85,Mic86,Ven91,RJL92,BSS96,Kiw08b}); quasi-Newton
  methods (\cite{DaF06,CMS14})
\item[{\sf (Notes)}] Numerical experiments ($n = 6.25 \cdot 10^6$) on random
  test problems, examining (a) the best number of initial (secant) Newton
  iterations, and (b) the performance against a primal pegging algorithm.
  Overall, breakpoint search is favourable on its own given a good initial
  mulitiplier estimate; otherwise, 3 or 4 iterations of the secant Newton
  method is a good initialization procedure.
\end{description}
\end{description}

\begin{description} 
\item[{\cite{FrG13}}] {\sc A.\ Frangioni and E.\ Gorgone}, {\em A library for
  continuous convex separable quadratic knapsack problems}
\begin{description}
\item[{\sf (Problem)}] $\phi_j(x_j) = \frac{q_j}{2}x^2_j - r_jx_j$; $q_j>0$;
  $g_j(x_j) = x_j$
\item[{\sf (Methodology)}] Compares a breakpoint algorithm with CPLEX and
  concludes that the breakpoint algorithm outperforms CPLEX
\item[{\sf (Citations)}] Applications in resource allocation and algorithms
  (\cite{Pat08})
\item[{\sf (Notes)}] Presents an open source library for the continuous convex
  separable quadratic knapsack problem and concludes that the library can be
  useful for further studies of the problem at hand.
 \end{description}
\end{description}

\begin{description} 
\item[{\cite{ZCCS13}}] {\sc T.\ Zhu, W.\ Chen, J.\ Chen, and W.\ Sun}, {\em
  Direct algorithm for continuous separable knapsack problem (in Chinese)}
\begin{description}
\item[{\sf (Problem)}] $\phi_j(x_j) = \frac{q_j}{2}x^2_j - r_jx_j$; $q_j>0$;
  $g_j(x_j) = a_jx_j$
\item[{\sf (Methodology)}] Pegging
\item[{\sf (Citations)}] Algorithms for the more general integer version of
  the problem; algorithms by Bretthauer and Shetty
  (\cite{BrS95,BrS02,BrS02b}); other specialized algorithms for the problem
  (\cite{BiH81,PaK90,RJL92,Ste01,Kiw08b})
 \item[{\sf (Notes)}] A detailed numerical example ($n = 8$); favourable
   comparisons with a Matlab solver ($n \in \{50, 100, 200\}$).
 \end{description}
\end{description}

\begin{description} 
\item[{\cite{Zha13}}] {\sc L.\ Zhang}, {\em A Newton-type algorithm for
  solving problems of search theory}
\begin{description}
\item[{\sf (Problem)}] $\phi_j(x_j) = -a_j(1 - {\rm exp}(-c_j x_j))$, $a_j,
  c_j > 0$; $X$ is a scaled unit simplex
\item[{\sf (Methodology)}] The KKT conditions, as defined in
  (\ref{eq:sep_conv_ineq_constr_kkt}), are relaxed into a system of non-smooth
  equations through the utilization of the Fischer--Burmeister (\cite{Fis92})
  smoothing function; a Newton-like algorithm is then employed for these
  equations for a sequence of values of the smoothing parameter. The algorithm
  is shown to asymptotically and superlinearly converge to the unique optimal
  solution
\item[{\sf (Citations)}] Survey (\cite{Pat08}); methodologies (\cite{Ste04});
  applications (\cite{Koo99})
\item[{\sf (Notes)}] Three sets of numerical experiments. Main experiment ($n
  \in [10^2, 10^4]$) on randomized problems; compares with the pegging
  algorithm in \cite{Ste04}, noting that the proposed algorithm is faster. The
  proposed methodology is however terminated based on a nonzero value of the
  Fischer--Burmeister smoothing function, whence the final solution need not
  be feasible or optimal. Second set of experiments on a problem taken from
  \cite{WiG69} ($n = 4$), showing no comparisons. Third experiment on data
  from the Bureau of Water Conservatory ($n = 5$), showing no comparisons.
 \end{description}
\end{description}

\begin{description} 
\item[{\cite{CMS14}}] {\sc R.\ Cominetti, W.\ F.\ Mascarenhas, and
  P.\ J.\ S.\ Silva}, {\em A Newton's method for the continuous quadratic
  knapsack problem}
\begin{description}
\item[{\sf (Problem)}] $\phi_j$ strictly convex quadratic; linear equality
\item[{\sf (Methodology)}] An approximative Newton method for the Lagrangian
  dual problem utilizing a secant globalization and variable fixing
\item[{\sf (Citations)}] Applications to resource allocation
  (\cite{BiH81,HoH95,BrS97}), multicommodity network flows
  (\cite{HKL80,NiZ92,ShM90}), Lagrangian relaxation using subgradient methods
  (\cite{HWC74}), quasi-Newton updates with bounds (\cite{CaM87}), semismooth
  Newton methods (\cite{FeM04}); related methods (\cite{DaF06,Kiw08b}), where
  the latter method is shown to be equivalent to the proposed one when there
  are only lower bounds {\sl or} lower bounds
\item[{\sf (Notes)}] Numerical experiments ($n \in [50 \cdot 10^3, 2 \cdot
  10^6]$) comparing the proposed method to a secant method (\cite{DaF06}),
  breakpoint search (\cite{Kiw08b}), and median search (\cite{Kiw08a} on
  problems with uncorrelated, weakly correlated, and correlated data. The
  proposed Newton method is overall better---about 30\% better on the larger
  instances. A further test is made on the classification problems described
  in \cite{DaF06} arising in the training in support vector machines; as the
  Hessian is non-diagonal, a projected gradient method is used, leading to
  subproblems of the type considered. Here, Newton's method is superior.
 \end{description}
\end{description}

\begin{description} 
\item[{\cite{WrR14}}] {\sc S.\ E.\ Wright and J.\ J.\ Rohal}, {\em Solving the
 continuous nonlinear resource allocation problem with an interior point
  method}
\begin{description}
\item[{\sf (Problem)}] $\phi_j$ and $g_j$ convex, and in $C^2$ on an open set
  containing $[l_j, u_j]$. Further, $\phi_j$ is decreasing on $[l_j, u_j]$ and
  $g_j$ is increasing on $[l_j, u_j]$, and $\sum_{j \in J} g_j(l_j) < b <
  \sum_{j \in J} g_j(u_j)$. Test instances include resource renewal
      [$\phi_j(x_j) := c_jx_j (e^{-1/x_j} - 1)$ for $x_j > 0$, $\phi_j(x_j) :=
        -x_j$ for $x_j < 0$, and $g_j$ linear], weighted $p$-norm over a ball
      [$\phi_j(x_j) := c_j(x_j - y_j)^p$ and $g_j(x_j) := |x_j|^r$ with $p, r
        \in \{2, 2.5, 3, 4\}$], sums of powers [$\phi_j(x_j) := c_j|x_j -
        y_j|^{p_j}$, and $g_j(x_j) := |x_j|^{r_i}$], convex quartic over a
      simplex [$\phi_j$ a fourth-power polynomial, and $g_j(x_j) := x_j$], and
      log-exponential [$\phi_j(x_j) := \ln [\sum_i {\rm exp}(a_{ij}x_j +
          d_{ij})]$, $g_j$ linear]
\item[{\sf (Methodology)}] A damped feasible interior-point Newton method for
  the solution of the KKT conditions
\item[{\sf (Citations)}] Survey (\cite{Pat08}); application to resource
  renewal (\cite{MeR00}); methodologies
  (\cite{Bru84,PaK90,RJL92,MeR00,Kiw08a})
\item[{\sf (Notes)}] The algorithm is introduced for problems where
  subsolutions are not available in closed form. Shows that the linear system
  defining the Newton search direction is solvable in $O(n)$ time. Numerical
  experiments ($n \in [10^2, 10^6]$) conclude that the interior point method
  wins over breakpoint search, often by an order of magnitude.
 \end{description}
\end{description}

\section{Breakpoint algorithms}\label{sec:rank}

Algorithms based on the Lagrangian relaxation of the explicit constraint
(\ref{problemb}) have an older history than the relaxation algorithms. This is
probably due to the fact that the relaxation algorithm quite strongly rests on
the Karush--Kuhn--Tucker (KKT) conditions, which did not become widely
available until the end of the 1940s and early 1950s with the work of F.\ John
\cite{Joh48}, W.\ Karush \cite{Kar39}, and H.~W. Kuhn and A.~W. Tucker
\cite{KuT51}. Lagrangian based algorithms have been present much longer and
the famous ``Lagrange multiplier method'' for equality constrained
optimization is classic in the calculus curriculum. Indeed, Lagrange
multiplier techniques for our problem (\ref{problem}) date back at least as
far as to \cite{Bec52}; see \cite{Pat08} for a survey of the history of the
problem.

Considering problem (\ref{problem}) and introducing the Lagrange multiplier
$\mu$ for constraint (\ref{problemb}), we obtain the following conditions for
the optimality of $\mathbf{x}^*$ in (\ref{problem}):
\begin{subequations}
\label{eq:sep_conv_ineq_constr_kkt}
\begin{align}
\mu ^* \geq 0, \quad g(\mathbf{x}^*) & \leq b, \quad \mu^* (
g(\mathbf{x}^*)-b) = 0,
\label{eq:sep_conv_ineq_constr_kkt1} \\
x_j^* & \in X_j, \qquad j \in J, \label{eq:sep_conv_ineq_constr_kkt2}
\end{align}
and
\begin{align} 
x_j^* = l_j, & \qquad \text{ if } \phi_j'(x_j^*) \geq - \mu^* g_j'(x_j^*),
\qquad j \in J, \label{eq:sep_conv_ineq_constr_kkt3} \\ x_j^* = u_j, & \qquad
\text{ if } \phi_j'(x_j^*) \leq - \mu^* g_j'(x_j^*), \qquad j \in
J, \label{eq:sep_conv_ineq_constr_kkt4} \\ l_j \leq x_j^* \leq u_j, & \qquad
\text{ if } \phi_j'(x_j^*) = - \mu^* g_j'(x_j^*), \qquad j \in J.
\label{eq:sep_conv_ineq_constr_kkt5}
\end{align}
\end{subequations}
For a fixed optimal value $\mu^*$ of the Lagrange multiplier the conditions
(\ref{eq:sep_conv_ineq_constr_kkt3})--(\ref{eq:sep_conv_ineq_constr_kkt5}) are
the optimality conditions for the minimization over $\mathbf{x} \in
\prod_{j=1}^n X_j$ of the Lagrangian function defined on $\prod_{j=1}^n X_j
\times \R_+$,
\[
L(\mathbf{ x}, \mu) := -b\mu + \sum_{j = 1}^n \{ \phi_j(x_j) + \mu g_j(x_j) \}.
\]
Given $\mu \geq 0$ its minimization over $\mathbf{ x} \in \prod_{j = 1}^n X_j$
separates into $n$ problems, yielding the Lagrangian dual function
\begin{equation}
\label{eq:sep_conv_gen_dual}
q(\mu) := -b \mu + \sum_{j = 1}^n \mathop{\rm minimum}_{x_j \in X_j} \, \{
\phi_j(x_j) + \mu g_j(x_j) \}.
\end{equation}

By introducing additional properties of the problem, we can ensure that the
function $q$ is not only concave but finite on $\R_+$ and
moreover differentiable there. Suppose, for example, that for each $j$,
$\phi_j(\cdot) + \mu g_j(\cdot)$ is weakly coercive on $X_j$ for every $\mu
\geq 0$ [that is, that either $X_j$ is bounded or that for every $\mu \geq 0$,
$\phi_j(x_j) + \mu g_j(x_j)$ tends to infinity whenever $x_j$ tends to $\pm
\infty$], and that $\phi_j$ is strictly convex on $X_j$. In this case
the derivative $q'$ exists on $\R_+$ and equals
\[
q'(\mu) = \phi'_j(x_j(\mu)) + \mu g'_j(x_j(\mu)),
\]
where $\mathbf{ x}(\mu)$ is the unique minimum of the Lagrange function
$L(\cdot , \mu)$ over $\prod_{j=1}^n X_j$. Thanks to this simple form of the
dual derivative, the maximum $\mu^*$ of $q$ over $\R_+$ can be characterized
by the complementarity conditions (\ref{eq:sep_conv_ineq_constr_kkt1}), and
the conditions (\ref{eq:sep_conv_ineq_constr_kkt}) are the primal--dual
optimality conditions for the pair of primal--dual convex programs.

If we assume that $\mu^* \neq 0$, we search for $\mu^* > 0$ such that
$q'(\mu^*) = 0$ [or, in other words, $g(\mathbf{ x}(\mu^*)) = b$], that is, we
need to solve a special equation in the unknown entity $\mu$, where the
function $q'$ is implicitly defined, but is known to be decreasing since $q$
is concave. This equation can of course be solved through the use of any
general such procedure [for example, {\sl bisection search} takes two initial
  values $\overline{\mu}$ and $\underline{\mu}$ with $q'(\overline{\mu}) < 0$
  and $q'(\underline{\mu}) > 0$, then iteratively cancels part of the initial
  interval given the sign of $q'$ at its midpoint $(\overline{\mu} +
  \underline{\mu})/2$], but the structure of $q'$ makes specialized algorithms
possible to utilize.

From the above optimality conditions for the Lagrangian minimization problem,
we obtain that
\begin{equation}\label{xmu}
x_j (\mu) = \begin{cases} l_j, & \text{if } \mu \geq \mu_j^l :=
  -\phi'_j(l_j)/g'_j(l_j), \cr u_j, & \text{if } \mu \leq \mu_j^u := -
  \phi'_j(u_j)/g'_j(u_j), \qquad j \in J. \cr x_j, & \text{if }
  \phi'_j(x_j) + \mu g'_j(x_j) = 0,
\end{cases}
\end{equation}
In a rudimentary algorithm we order these indices (or, breakpoints) $\mu_j^l$
and $\mu_j^u$ in an increasing (for example) order into $\{ \mu_1, \dotsc,
\mu_N \}$, where $N \leq 2n$ due to the possible presence of ties. Finding
$\mu^*$ then amounts to finding an index $\jmath^*$ such that that
$q'(\mu_{\jmath*}) > 0$ and $q'(\mu_{\jmath* + 1}) < 0$; then we know that
$\mu^* \in (\mu_{\jmath*}, \mu_{\jmath* +1})$ and $q'(\mu^*) = 0$.  Hence from
equation (\ref{xmu}), we know for all $j$ if $x_j^* = l_j$, $x_j^* = u_j$ or
$l_j < x_j^* < u_j$.  Now by fixing all variables $x_j^*$ that equal the
corresponding lower or upper bound we can ignore the bound constraint
(\ref{problemc}) and find an analytical solution of the problem.

Two decisions thus need to be made: how to find the index $\jmath^*$, and how
to perform the interpolation. Starting with the former, the easiest means is
to run through the indices in ascending or descending order to find the index
where $q'$ changes sign.  If we have access to the indices $\jmath^+$ and
$\jmath^-$ for which $q'(\mu_{\jmath^+}) > 0$ while $q'(\mu_{\jmath^-}) < 0$,
then we can choose the midpoint index, check the corresponding sign of $q'$,
and reduce the index set accordingly.  Given the sorted list, we can also find
this index in some randomized fashion.

As remarked above, algorithms such as bisection search can be implemented
without the use of the breakpoints, and therefore without the use of sorting,
as long as an initial interval can somehow be found; also general methods for
solving the equation $q'(\mu) = 0$, such as the secant method or regula falsi,
can be used even without an initial interval; notice however that $q \not\in
C^2$, whence a pure Newton method is not guaranteed to be well-defined.

While the sorting operation used in the ranking and bisection search methods
takes $O(n \log n)$ time, it is possible to lower the complexity by choosing
the trial index based on the {\sl median} index, which is found without the
use of sorting; the complexity of the algorithm is then reduced to $O(n)$. It
is not clear, however, that the latter must always be more efficient, since
the ``$O$'' definition calls for $n$ to be ``large enough''.

We also remark that in the case when the problem (\ref{problem})
arises as a subproblem in an iterative method, as the method converges the
data describing the problem will tend to stabilize. This fact motivates the
use of reoptimization of the problems, which most obviously can be done by
using the previous value of the Lagrange multiplier as a starting point and/or
utilizing the previous ordering of the breakpoints; in the latter case, the
$O(n \log n)$ sorting complexity will eventually drop dramatically.

In Section \ref{sec:equalityconstraint} we consider the breakpoint algorithm
for the equality problem (\ref{problemeq}). In Section \ref{sec:pegbreak} we
describe three pegging methods and in Sections
\ref{sec:medians}--\ref{sec:MB5} we apply these pegging methods to the
breakpoint algorithm. Finally, in Section \ref{sec:conbreak} we briefly discuss
the convergence and time complexity of the breakpoint algorithm.

\subsection{Equality constraints}\label{sec:equalityconstraint}
We now consider problem (\ref{problemeq}) where the inequality of the primal
constraint is replaced by an equality. For the problem to be convex, the
resource constraint (\ref{problemeqb}) has to be affine, i.e.,
$g(\mathbf{x}):=\sum_{j \in J} a_jx_j - b$. Beside the resource constraint,
the Lagrangian and the optimality conditions will take the same form as for
problem (\ref{problem}) but with one important difference; $\mu$ is
unrestricted in sign, whence the condition
(\ref{eq:sep_conv_ineq_constr_kkt1}) is replaced by ``$g(\mathbf{x}^*) = b$''.


\subsection{The pegging process}\label{sec:pegbreak}

The origin of the pegging process is found in the relaxation algorithm; see,
e.g., \cite{BiH81}. The purpose of pegging is to predict properties of the
primal variables in the optimal solution from an arbitrary dual value. In
Sections \ref{sec:taop}, \ref{sec:3set}, and \ref{sec:5set} we show how to
predict if an optimal primal variable value equals its lower or upper bound,
or is strictly within any of the bounds.

\subsubsection{2-sets pegging}\label{sec:taop}

If we can determine if a variable equals its lower bound at the optimal
solution then we add its variable index to a set $L$ and reduce the original
problem. Similarly, if we know that a variable equals its upper bound at the
optimal solution then we add the variable index to the set $U$. Using the sets
$L$ and $U$ when solving problem (\ref{problem}) or (\ref{problemeq}) will be
referred to as 2-sets pegging.

Assume that we have a lower limit $\underline{\mu}$ and an upper limit
$\overline{\mu}$ on the optimal dual value, that is, $\underline{\mu} \le
\mu^* \le \overline{\mu}$.  From (\ref{xmu}) we can define the sets
$L(\underline{\mu}) := \{ j \in J \mid \underline{\mu} \ge -\phi_j^\prime(l_j)
/ g_j^\prime(l_j) \}$ and $U(\overline{\mu}) := \{ j \in J \mid \overline{\mu}
\le -\phi_j^\prime(u_j) / g_j^\prime(u_j) \}$. Let $J^k := J \setminus \{
L(\underline{\mu}) \cup U(\overline{\mu}) \}$ and let $b^k := b - \sum_{j \in
  L (\underline{\mu})} g_j(l_j) -\sum_{j \in U(\overline{\mu})}
g_j(u_j)$. Hence we can define a subproblem of problem (\ref{problem}) as
follows:

\begin{subequations}
\begin{align}
        \minimize_x &\quad \phi(\mathbf{x}) := \sum _{j\in J^k}
        \phi_j(x_j), \\ \subto & \quad g(\mathbf{x}):= \sum
        _{j\in J^k} g_j(x_j) \le b^k, \label{problemb2} \\ 
        & \quad l_j \le x_j \le u_j, \quad j \in J^k. \label{problemc2}
\end{align}
\label{problem2}
\end{subequations}
Similarly we can define a subproblem of problem (\ref{problemeq}) as follows:
\begin{subequations}
\begin{align}
        \minimize_x &\quad \phi(\mathbf{x}) := \sum _{j \in J^k}
        \phi_j(x_j),\\ \subto & \quad g(\mathbf{x}):= \sum _{j\in
         J^k} a_jx_j = b^k, \label{problemeqb2} \\ 
          & \quad l_j \le x_j \le u_j, \quad j \in J^k.  \label{problemeqc2}
\end{align}
\label{problemeq2}
\end{subequations}

\noindent Consider problem (\ref{problem2}). Assuming that $\mu^*>0$, the
constraint (\ref{problemb}) has to be fulfilled with equality. For any given
dual variable $\mu^k$ we can determine the primal solution $\mathbf{x}^k(\mu
^k)$ of problem (\ref{problem2}) from condition (\ref{xmu}). We know from
Section \ref{sec:rank} that all optimality conditions
(\ref{eq:sep_conv_ineq_constr_kkt}) except the resource constraint
(\ref{problemb}) are satisfied. Moreover, we know that the resource constraint
has to be fulfilled with equality, in order for the solution to be
optimal. Substituting $\mathbf{x}^k$ into the resource constraint will be
referred to as explicit evaluation; this leaves us with three cases, namely
\begin{subequations}
\begin{align}
       \sum_{j\in J^k}g_j(x_j^k) & = b^k \label{caseopt}, \\
       \sum_{j\in J^k}g_j(x_j^k) & < b^k \label{caselow},  \mathrm{ or } \\
       \sum_{j\in J^k}g_j(x_j^k) & > b^k \label{casehigh}.
\end{align}
\end{subequations}
If (\ref{caseopt}) is fulfilled for $\mathbf{x}^k$ then all optimality
conditions are met and $\mathbf{x}^* = \mathbf{x}^k$.  Consider next the case
(\ref{caselow}); clearly $\mathbf{x}^k$ is not optimal but we know that
$\mathbf{x}^*$ is such that $\sum_{j \in J^k} g_j(x_j^*) > \sum_{j \in J^k}
g_j(x_j^k)$ since $\sum_{j \in J^k} g_j(x_j^*) = b^k$.  The function $g_j$ is
convex and differentiable and can increase in one interval and decrease in
another (consider, e.g., $g_j(x_j) = x^2_j$), which implies that no
predictions can be made of the size of $x_j^*$ relative to that of $x_j^k$.
Hence, we need $g_j$ to be monotone. For problem (\ref{problem}), Bretthauer
and Shetty \cite[Section 2]{BrS02} consider four cases equivalent to the
following:
\begin{description}
\item[Case 1:] For all $j\in J$, $g_j$ is decreasing and
  $\mu(x_j) := -\phi'_j (x_j) / g'_j (x_j)$ is increasing in $x_j$.

\item[Case 2:] For all $j\in J$, $g_j$ is increasing and
  $\mu(x_j)$ is decreasing in $x_j$.

\item[Case 3:] For all $j\in J$, $g_j$ is decreasing and
  $\mu(x_j)$ is decreasing in $x_j$.

\item[Case 4:] For all $j\in J$, $g_j$ is increasing and
  $\mu(x_j)$ is increasing in $x_j$.
\end{description}
If Case 3 or 4 holds it is possible to find a closed form of the optimal
solution to the problem (\ref{problem}), see \cite[Proposition
  10]{BrS02}. Considering problem (\ref{problemeq}), Cases 3 and 4 cannot
occur, since the resource constraint is affine, i.e., $\mu(x_j) := -\phi'_j
(x_j) / a_j$. Hence, only Cases 1 and 2 are of interest here.

Note that if $\mu(x_j)$ is increasing in $x_j$ then $x_j(\mu)$ is
nondecreasing and vice versa. This is an essential property for the validity
of the pegging process. We can indeed state the following proposition (a
similar one can be stated for problem (\ref{problemeq}), but without the
assumption $\mu^* > 0$):
 
\begin{proposition}[pegging for Cases 1 and 2]\label{pegglow}
Consider problem (\ref{problem}) and assume that $\mu^* > 0$.

(i) If Case 1 holds, and if (\ref{caselow}) holds, then $x_j^* = l_j$ for all
$j \in L(\mu^k)$.

(ii) If Case 1 holds, and if (\ref{casehigh}) holds, then $x_j^* = u_j$ for
all $j \in U(\mu^k)$.

(iii) If Case 2 holds, and if (\ref{caselow}) holds, then $x_j^* = u_j$ for
all $j \in U(\mu^k)$.

(iv) If Case 2 holds, and if (\ref{casehigh}) holds, then $x_j^* = l_j$ for
all $j \in L(\mu^k)$.

\end{proposition}

\begin{proof}
A proof of (i) is given; the proofs for (ii), (iii), and (iv) are
analogous. From (\ref{caselow}) we have that
\[ \sum_{j\in J^k}g_j(x_j(\mu^k)) < b^k.\] 
In Case 2, for all $j$ $g_j$ is increasing and $x_j(\mu)$ is nonincreaing,
which implies that $g_j(x_j(\mu))$ is nonincreasing in $\mu$ for all
$j$. Hence, we have that $\mu^k = \overline{\mu} \ge \mu^*$ which implies that
$x^k_j \le x^*_j$ for all $j$ since $x_j(\mu)$ is nonincreasing in $\mu$ for
all $j$. Hence, for $j \in U(\mu^k)$ we know that $x^{k+1}_j = u_j = x_j^*$,
i.e., we can peg $j \in U(\mu^k)$.
\end{proof}

\subsubsection{3-sets pegging}\label{sec:3set}

As in the 2-sets pegging principle of Section \ref{sec:taop} we determine if a
variable takes the value of the lower or upper bound at the optimal
solution. Additionally for the 3-sets pegging we determine if a variable
belongs to the open interval between the lower and upper bound, i.e., if
$x_j^* \in (l_j,u_j)$.  Assume that we know that $\mu^* \in (\mu_j^{u},
\mu_j^l)$; then it follows from (\ref{xmu}) that $x_j^* \in (l_j,u_j)$ and
there is no need to check if $x_j^k$ equals the lower or upper bound which
will reduce future calculations. 3-sets pegging is used for the quadratic
knapsack problem (\ref{quadratic}) in \cite[Section 3]{Kiw08a}.  The method
described in \cite[Section 3]{Kiw08a} can be generalized according to the
following proposition:

\begin{proposition}[relax primal variables from lower and upper
    bounds]\label{minkiwiel1}
Assume that Case 1 or 2 in Section \ref{sec:taop} holds and that for some
values of $\underline{\mu}$ and $\overline{\mu}$, $\underline{\mu} < \mu^* <
\overline{\mu}$ holds. If $\underline{\mu}, \overline{\mu} \in
         [\mu_j^l,\mu_j^u]$ holds for some $j \in \mathbf{J}^k$ then $l_j <
         x_j^* < u_j$.
\end{proposition} 

\begin{proof}
Assume that Case 2 holds, i.e., $\mu = -\phi'_j/g_j'$ is decreasing. (The
proof for Case 1 is analogous.) Assume that $\underline{\mu}, \overline{\mu}
\in [\mu_j^u,\mu_j^l]$ holds for some $j \in \mathbf{J}^k$. Since
$-\phi'_j/g_j'$ is decreasing we have that $\underline{\mu}, \overline{\mu}
\in [\mu_j^u,\mu_j^l]$ implies that $\mu^* \in (\mu_j^u,\mu_j^l)$, and from
(\ref{xmu}) it then follows that $l_j < x_j^* < u_j$.
\end{proof}

\subsubsection{5-sets pegging} \label{sec:5set}

As in the 3-sets pegging principle in Section \ref{sec:3set} we determine if a
variable takes the value of the lower or upper bound or if the variable
strictly belongs to the interval between the lower and upper bound. For 5-sets
pegging we also determine if a variable is larger than the lower bound or
smaller than the upper bound. 5-sets pegging for problem (\ref{problem}) is
used in \cite{WaW12}, generalizing a method from \cite{Kiw08a}. Assuming that
we know that $x_j^* < u_j$, there is no need to check if $x_j^k$ equals the
upper bound; this might reduce future calculations. The proof of the following
proposition follows from the monotonicity of $g_j$ and $x_j(\mu)$ (see
\cite{WaW12} for a proof).

\begin{proposition}[relax primal variables from lower or upper bound]
\label{wawa}
Assume that Case 1 or 2 in Section \ref{sec:taop} holds. Let $j \in
\mathbf{J}^k$. If $\mu^* < \overline{\mu} \leq \mu^l_j$, then $x_j^* >l_j$ and
if $\mu^*> \underline{\mu} \geq \mu^u_j$ then $x_j^* <u_j$.

\qed

\end{proposition}

\subsection{Algorithm: Median search of Breakpoints with 2-sets pegging 
(MB2)}\label{sec:medians}

Consider Case 1 or 2 in Section \ref{sec:taop} for problem
(\ref{problem}). Let $\rm{median}(\cdot)$ denote a function which provides the
median of a finite vector; let $\mu_m$ be the median breakpoint, and define
the total use of the resource due to variables that equal the lower and upper
bounds as $\beta_l := \sum_{ j\in \{ N \mid \mu_j^l \le \mu_m \} } g_j(l_j)$,
and $\beta_u := \sum_{j\in \{ N \mid \mu_j^u \ge \mu_m \}} g_j(u_j)$,
respectively.  In the spirit of \cite[Section 3.1]{BrS02b} and \cite[Algorithm
  3.1]{Kiw07}, we present the following algorithm:

\begin{tabular}{l}
\\ \textbf{Initialization:} \\ $\qquad$ Set $N := J$, $k:=1$, and $b^k := b$.
\\ $\qquad$ Compute breakpoints $\mu^l:= ( \mu^l_j )_{j \in N }$, $\mu^u:= (
\mu^u_j )_{j \in N }$ as in (\ref{xmu}), and let $\mu^k := \left( -\infty,
\left( \mu^l \right)^\intercal, \left( \mu^u\right)^\intercal, \infty
\right)^\intercal$.
\\ 
\textbf{Step 0} (check if $\mu = 0$ is optimal):
\\ 
$\qquad$ If $\sum_{j \in N} g_j (x_j(0)) \le b$, for $x_j(\mu)$ is determined 
from (\ref{xmu}), then $\mu^*=0$,  $x_j^* = x_j(0)$ for $j \in N$. Stop.    
\\ 
\textbf{Iterative algorithm:} 
\\ 
\textbf{Step 1} (stopping test): 
\\ 
$\qquad$ If $\mu^k = \emptyset$ then find $\mathbf{x}^*$ and $\mu^*$ from 
problem (\ref{problem}) relaxed from lower and upper bounds.
\\ 
$\qquad$ Otherwise, let $\mu_m := \rm{median}(\mu^k)$. 
\\ 
\textbf{Step 2} (compute explicit reference): 
\\ 
$\qquad$ Determine $ \delta := \sum_{ j\in \{ N \mid \mu_j^u < \mu_m < \mu_j^l
  \} } 
g_j(x_j(\mu_m))
+ \beta_u + \beta_l$, where $x_j(\mu)$ is determined from (\ref{xmu}).
\\ 
$\qquad$ If $\delta >b^k$, then go to Step 3.1. \\ $\qquad$ If $\delta < b^k$, 
then go to Step 3.2. \\ 
$\qquad$ Otherwise $(\delta = b^k)$ let $\mu^* := \mu_m$, find 
$\mathbf{x}^*$ from (\ref{xmu}), and stop.
\\
\textbf{Step  3.1} (update and fix lower bounds):
\\ 
$\qquad$ For all $j\in N$:  If $\mu_m \ge \mu_j^l$ then let
$N := N \setminus \{j\}$ and $x_j^* := l_j$. 
\\ 
$\qquad$ Let $\mu^{k+1} := ( \mu^k_j )_{ j \in \{ N \mid \mu_m < \mu^k_j \}
}$, 
$b^{k+1} := b^{k}- \beta_l$, and $k := k+1$. 
  \\ 
  $\qquad$ Go to Step 1.
   \\ 
  \textbf{Step 3.2} (update and fix upper bounds): 
 \\ 
  $\qquad$ For all $j\in N$: If $\mu_m \le \mu_j^u$ then let $N := 
N \setminus \{j\}$ and $x_j^* := u_j$.
   \\
    $\qquad$ Set $  \mu^{k+1}  := ( \mu^k_j )_{ j \in \{ N \mid \mu_m >
     \mu^k_j \}}$, 
$b^{k+1} := b^{k}- \beta_u$, and $k := k +1$.
  \\ $\qquad$ Go to Step 1.
  \\ \\
\end{tabular}

\noindent For problem (\ref{problemeq}), the algorithm is similar except Step
0 vanishes.

\subsection{Algorithm: Median search of Breakpoints with 3-sets pegging (MB3)}
\label{sec:rmm}

If we can determine if the value of a variable $x_j$ is strictly within the
bounds for all $\mu$ such that $\underline{\mu} < \mu < \overline{\mu}$, then
we don't have to check if $x_j$ violates the bounds when we determine $x_j$
from (\ref{xmu}) in future iterations (see Proposition \ref{minkiwiel1}). This
might save us some operations.  Define a set of indices for the lower and
upper limit being within the interval of the lower and upper breakpoint for a
variable: $M := \{ \, j \in N \mid \underline{\mu},\overline{\mu} \in
[\mu_j^l,\mu_j^u] \, \}$. From Proposition \ref{minkiwiel1} we have that if
$j\in M$ then $l_j < x_j^* < u_j$.  Hence if $j\in M$ we do not have to check
if $x_j$ violates the bounds in future iterations. But we should have in mind
that to determine if $j\in M$ requires some extra operations.  If we let the
initial values of the limits be $\underline{\mu} = -\infty$ and
$\overline{\mu} = \infty$, then $\underline{\mu},\overline{\mu} \notin
[\mu_j^l,\mu_j^u]$ and we note that there is no need to check if $j\in M$, as
long as $\underline{\mu} = -\infty$ or $\overline{\mu} = \infty$. Hence, to
avoid unnecessary operations we start by checking if $j\in M$ when
$\underline{\mu} > -\infty$ or $\overline{\mu} < \infty$; this is in contrast
to the algorithms in \cite[Section 3]{Kiw08a} and \cite{WaW12}.

We finally define the contribution to the resource constraint from the
variables including $M$: $\gamma(\mu) := \sum_{j\in M } g_j(x_j(\mu))$. Note
that the value of $\gamma$ depends on the value of the parameters in $\phi_j$,
$g_j$ and $\mu$. If we can separate the parameters from the multiplier $\mu$,
i.e., if $\gamma$ is additively and/or multiplicatively separable [that is,
  $\gamma (\mu , \cdot) = \gamma_1(\mu) \gamma_2(\cdot) + \gamma_3 (\cdot)$]
then the values of $\gamma_2$ and $\gamma_3$ can be calculated successively so
that no calculations are done more than once. Consider, for example, the
negative entropy function, $\phi_j(x_j) := x_j\log(x_j/a_j-1)$ and the
resource constraint function $g_j(x_j):=x_j$. Then, $\gamma(\mu,\mathbf{a})
:=\sum_{j\in M} g_j(x_j^{k}(\mu)) = \sum_{j\in M} a_je^{-\mu} =
\gamma_2(\mathbf{a})\gamma_1(\mu)$, where $\gamma_1(\mu) = e^{-\mu}$ and
$\gamma_2(\mathbf{a}) = \sum_{j\in M} a_j$, i.e., we update $\gamma_2$ in
Steps 3.1 and 3.2 such that if $\underline{\mu},\overline{\mu} \in
[\mu_j^u,\mu_j^l]$ then $\gamma_2 := \gamma_2 + a_j$. For the quadratic
knapsack problem, this approach is applied in \cite[Section 3]{Kiw08a}.  We
next present an algorithm that applies the usage of $M$; Steps 0 and 1 are
similar to the algorithm MB2 in Section \ref{sec:medians}:

\begin{tabular}{l}
\\ \textbf{Initialization:} 
\\ 
$\qquad$ Set $N := J$, $k:=1$, $b^k := b$, $\gamma := 0$, $M := \emptyset$,  
$\underline{\mu} := -\infty$, and $\overline{\mu} := \infty$. 
\\
$\qquad$ Compute breakpoints $\mu^l:= ( \mu^l_j )_{j \in N }$, $\mu^u:= 
( \mu^u_j )_{j \in N }$ as in (\ref{xmu}), and let $\mu^k :=
\left( \begin{matrix} \mu^l \\ \mu^u \end{matrix} \right)$.
\\ 
\textbf{Iterative algorithm:} 
\\ 
\textbf{Step 2} (compute explicit reference):
\\ $\qquad$ Determine $ \delta := \sum_{j\in \{ N \mid \mu_j^u < \mu_m < \mu_j^l \}}
g_j(x_j(\mu_m)) + \gamma(\mu_m) + \beta_u + \beta_l$, 
\\
$\qquad$ where $x_j(\mu)$ is determined from (\ref{xmu}). 
\\
$\qquad$ If $\delta >b^k$, then go to Step 3.1. 
\\ 
$\qquad$ If $\delta < b^k$, then go to Step 3.2. 
\\ 
$\qquad$ Otherwise $(\delta = b^k)$ set $\mu^* := \mu_m$; find optimal 
$\mathbf{x}^*$ from (\ref{xmu}), and stop.
\\ 
\textbf{Step 3.1} (update and fix lower bounds): 
\\ 
$\qquad$ Let $\underline{\mu}=\mu_m$.
\\
$\qquad$ For $j\in N$: If $\underline{\mu} \ge \mu_j^l$ then let $N
:= N \setminus \{j\}$ and $x_j^* := l_j$. \\ $\qquad \qquad\quad\; $
$\quad$  If $\underline{\mu},\overline{\mu} \in [\mu_j^u,\mu_j^l]$
then let $N := N
\setminus \{j\}$ and $M := M \cup \{j\}$; update $\gamma_{1}$ and $\gamma_{2}$.
\\ 
$\qquad$ Let $\mu^{k+1} := ( \mu^k_j )_{ j \in \{ N \cup M \mid \mu_m < \mu^k_j \}
}$, $b^{k+1} := b^{k}- 
\beta_l$, and $k := k+1$. 
\\ $\qquad$ Update $\gamma$ and go to Step 1. 
\\ 
\textbf{Step 3.2} (update and fix upper bounds):
 \\ 
 $\qquad$ Let $\overline{\mu} := \mu_m$.
 \\
 $\qquad$ For
$j\in N$: If $\overline{\mu} \le \mu_j^u$ then let $N := N \setminus
\{j\}$ and $x_j^* := u_j$. 
\\ $\qquad \qquad\quad\; $ $\quad$ If  $\underline{\mu},\overline{\mu} \in
   [\mu_j^u,\mu_j^l]$ then let $N := N \setminus \{j\}$ and
$M := M \cup \{j\}$; update $\gamma_{1}$ and $\gamma_{2}$. 
\\ 
$\qquad$ Let $\mu^{k+1}  := ( \mu^k_j )_{ j \in \{ N \cup M \mid \mu_m > \mu^k_j
  \}}$, $b^{k+1} := b^{k} - \beta_u$, and $k := k +1$. 
\\ $\qquad$ Go to Step 1. 
\\
\end{tabular}

\begin{remark} If $\gamma$ is not separable then the updates of $\gamma_1$
  and/or $\gamma_2$ in Steps 3.1 and 3.2 vanish.
\end{remark}

\subsection{Algorithm: Median search of Breakpoints with 5-set pegging
  (MB5)}\label{sec:MB5}

As in the algorithm MB2 in Section \ref{sec:medians} we peg the variables
whose values equal the lower or upper bounds, and as in MB3 we determine if a
variable is strictly within the bounds. Further, we determine if a variable is
smaller than the upper bound or larger than the lower bound as in Section
\ref{sec:5set}. Define a set $L_-$ of indices where the optimal solution is
known to be strictly smaller than the upper bound in the optimal solution,
i.e., $L_- := \{j \in {N} \mid x_j^*<u_j \}$, and similarly define a set $U_+$
of indices where the variable is known to be larger than the lower bound in
the optimal solution, i.e., $U_+:=\{j \in {N} \mid x_j^*>l_j\}$. Define,
respectively, the total use of the resource due to variables whose values
equal the lower and upper bounds as $\beta_l := \sum_{j \in \{ N \cup L_- \mid
  \mu_j^l \le \mu_m\} } g_j(l_j)$ and $\beta_u := \sum_{j\in \{ N \cup U_+
  \mid \mu_j^u \ge \mu_m \} } g_j(u_j)$. We present an algorithm that applies
5-sets pegging where Steps 0 and 1 are similar to the algorithm MB2 in Section
\ref{sec:medians}:

\begin{tabular}{l}
\\ 
\textbf{Initialization:} 
\\ 
$\qquad$ Set $N := J$, $k:=1$, $b^k := b$, $\gamma :=0$, and 
$M = L_- = U_+  := \emptyset$,
$\underline{\mu} := -\infty$, and $\overline{\mu} := \infty$. 
\\
$\qquad$ Compute breakpoints $\mu^l:= ( \mu^l_j )_{j \in N }$, $\mu^u:= 
( \mu^u_j )_{j \in N }$ as in (\ref{xmu}), and let $\mu^k :=
\left( \begin{matrix} \mu^l \\ \mu^u \end{matrix} \right)$. 
\\ 
\textbf{Iterative algorithm:}
\\ 
\textbf{Step 2} (compute explicit reference):
\\ $\qquad$ Determine $ \delta := \sum_{j\in \{ N \cup U_+ \cup L_-  \mid
  \mu_j^u < \mu_m < \mu_j^l \} }
g_j(x_j(\mu_m)) + \gamma(\mu_m) + \beta_u + \beta_l$, 
\\
$\qquad$
where $x_j(\mu)$ is determined from (\ref{xmu}). 
\\
$\qquad$ If $\delta >b^k$, then go to Step 3.1. 
\\ 
$\qquad$ If $\delta < b^k$, then go to Step 3.2. \\ $\qquad$
Otherwise $(\delta = b^k)$ set $\mu^* := \mu_m$ and find optimal 
$\mathbf{x}^*$ from (\ref{xmu}), and stop.
\\ 
\textbf{Step 3.1} (update and fix lower bounds):
\\
$\qquad$ Let $\underline{\mu} := \mu_m$
\\ 
$\qquad$ For $j\in N$: If $ \underline{\mu} \ge \mu_j^l$ then let $N
:= N \setminus \{j\}$  and $x_j^* := l_j$. 
\\
$\qquad$ For $j \in L_-$: If $\underline{\mu} \ge \mu_j^l$ then let $L_- := 
L_- \setminus \{j\}$ and $x_j^* := l_j$. 
\\ 
$\qquad$ For $j\in N$: If $\underline{\mu} \ge \mu_j^u$ then let $N
:= N \setminus \{j\}$ and $L_- := L_-  \cup \{j\}$. 
\\
$\qquad$ For $j\in U_+$: If $\underline{\mu} \in [\mu_j^u,\mu_j^l]$ and 
$\overline{\mu}\ge\mu^u_j$
then  let $U_+ := U_+ \setminus \{j\}$ and $M := M \cup \{j\}$; update 
$\gamma_{1}$, $\gamma_{2}$. 
\\ 
$\qquad$ Let $\mu^{k+1} := ( \mu^k_j )_{ j \in \{ N \cup M \cup L_- \cup U_+ 
\mid \mu_m < \mu^k_j \}}$, 
$b^{k+1} := b^{k}-
\beta_l$, and $k := k+1$. 
\\ 
$\qquad$ Go to Step 1. 
\\
\textbf{Step 3.2} (update and fix upper bounds):
\\ 
$\qquad$ Let $\overline{\mu} := \mu_m$.
\\
$\qquad$ For $j\in N$: If $ \overline{\mu} \le \mu_j^u$ then let 
$N := N \setminus \{j\}$ and $x_j^* := u_j$. 
\\
$\qquad$ For $j \in U_+$: If $\overline{\mu} \le \mu_j^u$ then let 
$U_+ := U_+ \setminus \{j\}$ and $x_j^* := u_j$. 
\\ 
$\qquad$ For $j\in N$: If $\overline{\mu} \le \mu_j^l$ then let 
$N := N \setminus \{j\}$ and $U_+ := U_+  \cup \{j\}$. 
\\
$\qquad$ For $j\in L_-$: If $\overline{\mu} \in [\mu_j^u,\mu_j^l]$ and 
$\underline{\mu}\le\mu^l_j$
then let $L_- := L_- \setminus \{j\}$ and $M := M \cup \{j\}$; 
update $\gamma_{1}$, $\gamma_{2}$.
\\ 
$\qquad$ Let $\mu^{k+1}  := ( \mu^k_j )_{ j \in \{ N \cup M \cup L_- 
\cup U_+ \mid \mu_m > \mu^k_j \}}$, 
$b^{k+1} := b^{k}-
\beta_u$, and $k := k+1$. 
\\ 
$\qquad$ Go to Step 1. 
\\ 
\\
\end{tabular}

\noindent Note that when we determine $x_j(\mu_m)$ for $j \in U_+$ we do not
need to check if $x_j=l_j$, and for $j\in L_-$ we do need to check if
$x_j=u_j$. This will in some cases save us some operations. Further when we
determine if $x_j \in M$ we only need to check if this holds for $j \in U_+$
or $j \in L_-$ depending on if we update the lower or upper bound of the dual
variable. This differs from the algorithm in \cite{WaW12} since that algorithm
finds $M$ from $U_+ \cup L_-$. Note that we make use of information from
earlier iterations when updating $\delta$. When updating $\delta$ in
\cite{WaW12} information from earlier iteration is negligible hence some
operations are repeated.

\begin{remark}
Consider iteration $k$. In Step 3.1 we only need to check if $\mu_m \in
[\mu_j^u,\mu_j^l]$ and $\overline{\mu}\ge\mu^u_j$ for $j\in U_+$ since we know
from earlier iterations that if $j \in N$ then $\overline{\mu} \nleq
\mu^l_j$. Similar holds for Step 3.2.
\end{remark}

\subsection{Convergence and time complexity of breakpoint 
algorithms}\label{sec:conbreak}

Similar to \cite{Kiw08a} and \cite{WaW12} it is possible to show that the
breakpoint algorithms converge to the optimal solution.  Consider the time
complexity for algorithm MB2, MB3 and MB5. Assuming that the median function
$\rm{median(\cdot)}$ is linear, we have $C_i n$ operations in each of the
Steps 0 through 3.2 for some constants $C_i$ for $i \in \{ 0; 1; 2; 3.1;
3.2\}$ corresponding to Step 0 through 3.2 respectively.  Since we use a
median search function the number of iterations will terminate in $\log_2 (n)$
iterations (in the worst case). Hence the time complexity of the algorithm is
$O (n \log_2 (n))$. However in \cite{Bru84} a proof for a time complexity of
$O (n)$ is given for the breakpoint algorithm solving the quadratic knapsack
problem (\ref{quadratic}).

\section{Relaxation algorithms}\label{sec:relaxation}

In a relaxation algorithm for the problem (\ref{problem}) we iteratively solve
the problem relaxed from constraints (\ref{problemc}), i.e., we solve the
following problem:
\begin{subequations}
\begin{align}
        \minimize_x &\quad \phi(\mathbf{x}) := \sum_{j \in J^k} \phi_j(x_j),
        \\ \subto & \quad g(\mathbf{x}):= \sum_{j \in J^k} g_j(x_j) \le b.
\end{align}
\label{problemrelax}
\end{subequations}

\noindent From problem (\ref{problemrelax}) we obtain a solution
$\hat{\mathbf{x}}$. Together with $\hat{\mathbf{x}}$ we also obtain an
estimate $\hat{\mu}$ of the multiplier value $\mu^*$ from the optimality
condition. Then we adjust the solution $\hat{\mathbf{x}}$ for constraints
(\ref{problemc}) by determining $x_j$ from
\begin{equation}\label{xprim}
x_j := \begin{cases} l_j, & \text{if } \hat{x}_j \leq l_j, 
\cr u_j, & \text{if } \hat{x}_j \geq u_j, 
\cr \hat{x}_j, & \text{if } l_j<\hat{x}_j<u_j.
\end{cases}
\end{equation}

At the beginning of the algorithm we must determine whether constraint
(\ref{problemb}) is satisfied with equality at an optimal solution. From the
optimality condition (\ref{eq:sep_conv_ineq_constr_kkt1}) we have that if the
inequality constraint ({\ref{problemb}}) is satisfied then $\mu^*=0$. Hence,
for $\hat{\mu} = 0$ we find $\hat{x}_j$ by solving $\phi_j(\hat{x}_j) +
\hat{\mu} g_j(\hat{x}_j)= 0 $; if $\sum_{j \in J} g_j(x_j) \le b$, where the
value of $x_j$ is determined from ({\ref{xprim}}), then we have found the
optimal solution to problem (\ref{problem}).  Otherwise, we know that $\mu^* >
0$ and that the inequality constraint (\ref{problemb}) can be regarded as an
equality.  Hence, we solve the problem (\ref{problemrelax}), obtaining a
solution $\hat{\mathbf{x}}$. Let
\[
L(\hat {\mathbf{x}}) := \{ \, j \in J^k \mid \hat{x}_j \le l_j \, \},
\qquad U(\hat {\mathbf{x}}) := \{ \, j \in J^k \mid \hat{x}_j \ge u_j
\, \}
\]
denote the sets of variables that are either out of bounds at $\hat{\mathbf{
    x}}$ or equal a lower, respectively an upper, bound. 

In order to simplify the remaining discussion, we consider Case 2 in Section
\ref{sec:taop}, i.e., $\mu(x_j)$ to be monotonically decreasing and $g_j$ is
increasing; Case 1 is treated analogously.  Calculate the total
deficit and excess at $\hat{\mathbf{ x}}$, respectively, as
\begin{equation}\label{nabladelta}
\nabla := \sum_{j \in L} (g_j(l_j) - g_j(\hat{x}_j)), \qquad
\Delta := \sum_{j \in U} (g_j(\hat{x}_j) - g_j(u_j)).
\end{equation}
Now, if $\Delta > \nabla$ then we set $x_j^* = u_j$ for $j \in U(\hat{\mathbf
  {x}})$; if $\Delta < \nabla$ we set $x_j^* = l_j$ for $j \in L(\hat{\mathbf{
    x}})$; otherwise $\Delta = \nabla$ and we have found the optimal
solution. If $\Delta \neq \nabla$, then we reduce the problem by removing the
fixed variables, and adjusting the right-hand side of the constraint
(\ref{problemb}) to reflect the variables fixed. If any free variables are
left, we re-solve problem (\ref{problemrelax}) and repeat the procedure,
otherwise we have obtained an optimal solution.

The rationale behind this procedure is quite simple and natural: Suppose that
$\Delta > \nabla$ holds. We have that $\hat{\mu} = -
\phi_j'(\hat{x}_j)/g_j'(\hat{x})$ for $j \in J^k \setminus \{ L \cup U
\}$. Let $s \in U(\hat{\mathbf{ x}})$ and $i \in J^k \setminus U(\hat{\mathbf{
    x}})$. Since the functions $-\phi_j' / g_j'$ are decreasing, it follows
that
\[
- \frac{\phi_s'(u_s)}{g'_s(u_s)} \geq
- \frac{\phi_s'(\hat{x}_s)}{g'_s(\hat{x}_s)} = \hat{\mu} =
- \frac{\phi_i'(\hat{x}_i)}{g_i'(\hat{x}_i)} \geq
- \frac{\phi_i'(u_i)}{g'_i(u_i)}.
\]
Denote by $b_+$ the right-hand side in the following iteration given that
$\Delta > \nabla$ holds: $b_+ := b - \sum_{j \in U(\hat{ x})} g_j(u_j)$. Also
let $(\hat{\mathbf{x}}', \hat{\mu}')$ denote a pair of relaxed optimal
primal--dual solutions in the following iteration. We must have that $\hat\mu'
\leq \hat\mu$, since
\[
\sum_{j \in J^k \setminus U(\hat{x}) } g_j(\hat{x}_j) = b - \sum_{j \in U
  (\hat{x})} g_j(\hat{x}_j) \leq b - \sum_{j \in U (\hat{ x})} g_j(u_j) = b_+
= \sum_{j \in J^k \setminus U(\hat{ x}) } g_j(\hat{x}_j');
\]
hence, for at least one $j \in J \setminus U(\hat{\mathbf{x}})$ we have that
$\hat{x}_j' \geq \hat{x}_j$, and therefore,
\[
\hat{\mu}' = - \frac{\phi_j'(\hat{x}_j')}{g_j'(\hat{x}_j')} \leq -
\frac{\phi_j'(\hat{x}_j)}{g_j'(\hat{x})} = \hat{\mu}
\]
follows. This derivation was first described by Bitran and Hax \cite{BiH81}.

Since in each iteration at least one variable is pegged to an optimal value,
the algorithm is clearly finite.  The most serious disadvantage of the
algorithm may be the requirement that the problem without the variable bounds
present must have an optimal solution. The computational efficiency of this
method is also determined by whether or not it is possible to provide an
explicit formula for each $\hat{x}_j$ in terms of the multiplier.


\subsection{Explicit/Implicit evaluation}\label{sec:peme}

For the breakpoint algorithm we determine $\mathbf{x}$ from (\ref{xmu}) for an
arbitrary multiplier $\mu$; then we explicitly evaluate the optimality of
$\mathbf{x}$ by substituting it into the resource constraint
({\ref{problemb}}). The explicit evaluation leaves us with one out of three
possible scenarios: (\ref{caseopt})--(\ref{casehigh}). For the relaxation
algorithm the traditional method to evaluate a solution to the problem
(\ref{problemrelax}) is to calculate the total deficit and excess at
$\mathbf{\hat{x}}$ (see Section \ref{sec:relaxation}). Also this evaluation
leaves us with 3 possible scenarios, namely
\begin{subequations}\label{relaxbit2}
\begin{align}
       \Delta & = \nabla \label{caseopt3}, \\
       \Delta & > \nabla \label{caselow3}, \\
       \Delta & < \nabla \label{casehigh3}. 
\end{align}
\end{subequations}
 
Evaluating the optimality from (\ref{caseopt3})--(\ref{casehigh3}) will be
referred to as an {\em implicit evaluation}. For the relaxation algorithm we
next show that the implicit and explicit evaluations are equivalent.
Propositions \ref{prop1brs} and \ref{prop2brs} below state the relations
between explicit and implicit evaluation. The proof of Proposition
\ref{prop2brs} is similar to the proof of Proposition \ref{prop1brs}.

\begin{proposition}[relation between explicit and implicit evaluation for 
$g_j$ monotonically increasing]\label{prop1brs} If for all $j\in J$ $g_j$ is
  monotonically increasing, then the explicit evaluation
  (\ref{caseopt})--(\ref{casehigh}) is equivalent to the implicit evaluation
  (\ref{caseopt3})--(\ref{casehigh3}), i.e., (\ref{caseopt3})
  $\Longleftrightarrow$ (\ref{caseopt}), (\ref{caselow3})
  $\Longleftrightarrow$ (\ref{caselow}), and (\ref{casehigh3})
  $\Longleftrightarrow$ (\ref{casehigh}).
\end{proposition}

\begin{proof}
For all $j \in J$, let $\hat{x}_j$ be the solution to ({\ref{problemrelax}}).
Let $ \nabla$ and $\Delta$ be defined as in (\ref{nabladelta}). We
have from (\ref{xprim}) that
\begin{align*}
\sum_{j\in J}g_j(x_j) = & \sum_{j\in J \setminus \{
  U \cup L \} } g_j(\hat{x}_j) + \sum_{j\in
  L} \left \{ g_j(l_j)-g_j(\hat{x}_j) + g_j(\hat{x}_j)\right
\} \nonumber 
\\ 
& + \sum_{j\in U}\left \{ g_j(u_j) -
g_j(\hat{x}_j) + g_j(\hat{x}_j) \right \} \nonumber \\ = &
\sum_{j\in J \setminus \{ U \cup L \} }
g_j(\hat{x}_j) + \sum_{j\in L} \left \{
g_j(l_j)-g_j(\hat{x}_j)\right \} + \sum_{j\in L}
g_j(\hat{x}_j) \nonumber \\ & - \sum_{j\in U} \left \{ g_j(u_j)
- g_j(\hat{x}_j) \right \} + \sum_{j\in U} g_j(\hat{x}_j)
\nonumber 
\\ 
= & \sum_{j\in J} g_j(\hat{x}_j) + \nabla - \Delta \nonumber 
\\ 
= & \; b + \nabla - \Delta, \nonumber
\end{align*}
where the last equality follows from the fact that $\hat{\mathbf{x}}$ is the
solution to the relaxed problem. Hence, $\hat{\mathbf{x}}$ must satisfy the
resource constraint. We know that $\Delta,\nabla \ge 0$ since $g_j$ is
increasing. Hence, if $\Delta = \nabla$ then (\ref{caseopt}) holds, if $\Delta
>\nabla$ then (\ref{caselow}) holds, and if $\Delta < \nabla$ then
(\ref{casehigh}) holds.
\end{proof}

\begin{proposition}[relation between explicit and implicit evaluation for 
$g_j$ monotonically decreasing]\label{prop2brs} If $g_j$ is monotonically
  decreasing for all $j\in J$, then the explicit evaluation
  (\ref{caseopt})--(\ref{casehigh}) is equivalent to the implicit evaluation
  (\ref{caseopt})--(\ref{casehigh}), i.e., (\ref{caseopt3})
  $\Longleftrightarrow$ (\ref{caseopt}), (\ref{casehigh3})
  $\Longleftrightarrow$ (\ref{caselow}), and (\ref{caselow3})
  $\Longleftrightarrow$ (\ref{casehigh}). 
\end{proposition}

\subsection{Primal/Dual evaluation of boundaries}\label{sec:primDualEval}
In the breakpoint algorithms in Section \ref{sec:rank}, we use the dual
variable to determine if $x_j$ equals a bound or if it lies in between the
bounds; see (\ref{xmu}). In relaxation algorithms the traditional way to solve
the problem (\ref{problem}) is to determine the primal optimal solution
$\hat{x}_j$ of the relaxed problem (\ref{problemrelax}) and then simply check
if $\hat{x}_j$ is within the lower and upper bounds; see (\ref{xprim}).  An
alternative is to find the optimal dual variable $\hat{\mu}$ of the relaxed
problem (\ref{problemrelax}) and then (similar to the breakpoint algorithm in
Section \ref{sec:rank}) determine the primal variables from (\ref{xmu}). Of
course this requires us to determine the breakpoints. However, for the
variables that violate the bounds (\ref{problemc}) we don't have to determine
$x_j$ from the relation $\phi_j(x_j) + \mu g_j(x_j)= 0$.  Hence, instead of
evaluating $L(\hat{\mathbf{x}})$ and $U(\hat{\mathbf{x}})$ from the primal
variable as in Section \ref{sec:relaxation}, we can evaluate $L(\hat{\mu})$
and $U(\hat{\mu})$ from the dual variable as in Section \ref{sec:taop}.

Define $\Phi_j(x_j) := \phi_j^\prime(x_j) / g_j^\prime(x_j)$ and assume that
there exists an inverse to $\Phi_j$.  It follows from the optimality
conditions (\ref{xmu}) that $\Phi_j(\hat{x}_j(\hat\mu)) = -\hat\mu$,
$\Phi_j(l_j) = \phi_j^\prime(l_j) / g_j^\prime(l_j)$ and $\Phi_j(u_j) =
\phi_j^\prime(u_j) / g_j^\prime(u_j)$. Hence we have that:
\[
L(\hat {\mu}) := \{ \, j \in J \mid \hat{\mu} \ge -\phi_j^\prime(l_j) /
g_j^\prime(l_j) \, \} \Longleftrightarrow L(\hat {\mathbf{x}}) := \{ \, j \in
J \mid \hat{x}_j \le l_j \, \},
\]
and
\[
U(\hat {\mu}) := \{ \, j \in J \mid \hat{\mu} \le -\phi_j^\prime(u_j) /
g_j^\prime(u_j) \, \} \Longleftrightarrow U(\hat {\mathbf{x}}) := \{ \, j \in
J \mid \hat{x}_j \ge u_j \, \}.
\]

\subsection{Implementation choices for the relaxation algorithm}
\label{sec:pdei}

Considering the performance of a relaxation algorithm two decisions need to be
made: should the relaxed solution of the problem be evaluated implicitly from
$\nabla$ and $\Delta$ or explicitly from $\sum g_j$ (see Section
\ref{sec:peme})? Should the algorithm solve the primal or dual relaxed problem
(see Section \ref{sec:primDualEval})? An overview of the relaxation algorithm
and its possible realizations is shown in Figure \ref{relaxationOverview}. The
leftmost path in the figure is the classic primal relaxation algorithm of
Bitran and Hax \cite{BiH81}, the rightmost path in the figure is implemented
in \cite{Ste01}, and the algorithm that applies the right path in Step 1 and
the left path in Step 2 is implemented in \cite{KiW12}. Beside these two paths
no other paths have been explored. Our intention is to evaluate the
theoretically and practically best performing paths. Since no earlier studies
have applied 3- or 5-sets pegging for the relaxation algorithm, our intention
is to apply these two more sophisticated pegging methods.

\begin{figure}[h!]
  \centering
\psfrag{C}{\small \bf ~~Check if inequality feasible}
\psfrag{F}{\small \bf ~~~~~~~~~~Find breakpoints}
\psfrag{P}{\small \bf ~~~~Solve relaxed primal problem}
\psfrag{D}{\small \bf ~~~~Solve relaxed dual problem}
\psfrag{I}{\small \bf ~~~~~~~~~~~Implicit evaluation}
\psfrag{E}{\small \bf ~~~~~~~~~~~Explicit evaluation}
\psfrag{A}{\small \bf \!\!\!\!\!\! Apply 2-, 3-, or 5-sets pegging. 
                  If optimal stop, else go to Step 1}
\psfrag{a}{\small \bf Step 1}
\psfrag{b}{\small \bf Step 2}
\psfrag{c}{\small \bf Step 3}
  \includegraphics[width=12cm]{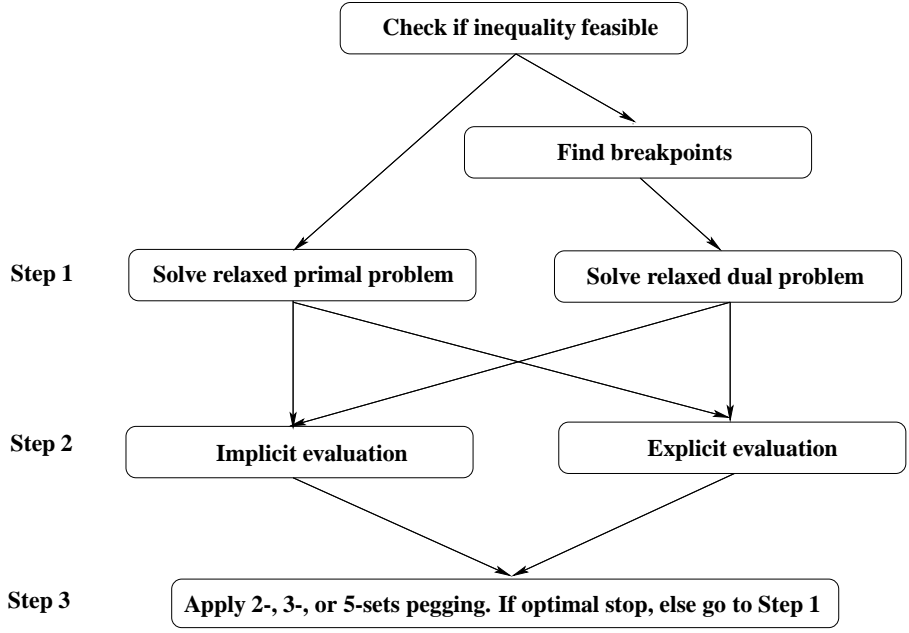}
  \caption{The diagram shows different possibilities to consider when
    constructing a relaxation algorithm.} \label{relaxationOverview}
\end{figure}


In Section \ref{sec:bih} we present an algorithm corresponding to the leftmost
path in Figure \ref{relaxationOverview}; in Section \ref{sec:mod2} we present
an algorithm which utilizes the rightmost path in Step 1 of Figure
\ref{relaxationOverview} and then changes to the left path in Step 2 of the
figure; in Section \ref{sec:ste} an algorithm corresponding to the rightmost
path of Figure \ref{relaxationOverview} is presented; and in Section
\ref{sec:bh4} an algorithm that utilizes the rightmost path in Step 1 of
Figure \ref{relaxationOverview} and then utilizes the theoretically best path
in Step 2 is presented. All algorithms in Sections
\ref{sec:bih}--\ref{sec:bh4} utilize 2-sets pegging. In Section
\ref{sec:tandf} we describe how 3- and 5-sets pegging can be utilized in
these.

\subsubsection{Algorithm: Primal determination with Implicit evaluation of 
the Relaxed problem (PIR2)}\label{sec:bih}

We first assume that Case 1 in Section \ref{sec:taop} holds for problem
(\ref{problem}). Define parameters to calculate the total deficit and excess,
as $\alpha^k_+ := \sum_{j \in U^k} g_j(\hat{x}_j)$, $\alpha^k_- := \sum_{j \in
  L^k} g_j(\hat{x}_j)$, $\beta^k_- := \sum_{j \in L^k} g_j(l_j)$ and
$\beta^k_+ := \sum_{j \in U^k} g_j(u_j)$. Hence we have that $\Delta^k =
\alpha^k_+ -\beta^k_+ $ and $\nabla^k = \beta^k_-- \alpha^k_-$. In the next
iteration when we reduce the resource $b^k$ we hence don't have to
re-calculate the part of the pegged variables which define $\beta^k_+$ or
$\beta^k_-$ (see Steps 3.1 and 3.2).  We present an algorithm for problem
(\ref{problem}) which is similar to the algorithm in Section 2 of
\cite{BrS02}:

\begin{tabular}[h]{l}
\\ 
\textbf{Initialization: } Set $k:=1, \; J^k := J$,
and $b^k := b$. 
\\ 
\textbf{Step 0} (check if $\mu = 0$ is optimal):
\\ 
$\qquad$ Let $\mu=0$ and find the solution $\mathbf{\hat{x}}$ to the relaxed
problem (\ref{problemrelax}), i.e., solve $\phi_j'(\hat{x}_j) + \mu
g_j'(\hat{x}_j)$, $j \in J^k$. 
\\
$\qquad$ If $\sum_{j \in J^k} g_j (x_j) \le b$, for $x_j$ determined from
(\ref{xprim}), then $\mu^*=0$,  $x_j^* = x_j$ for $j \in J^k$, and stop.  
\\
\textbf{Iterative algorithm: } 
\\ 
\textbf{Step 1} (solve relaxed primal problem):
\\ $\qquad$ For $j\in J^k$, find $\hat{x}_j^k$ by solving the relaxed problem
(\ref{problemrelax}).  
\\
$\qquad$ Set $L:=\emptyset$ and $U:=\emptyset$.
\\ 
\textbf{Step 2} (implicit evaluation): \\ $\qquad$
Determine $U(\hat{\mathbf{x}}^k)$ and $L(\hat{\mathbf{x}}^k)$ while computing
$\Delta^k := \alpha^k_+ -\beta^k_+ $ and $\nabla^k := \beta^k_-- \alpha^k_-$.  
\\ 
$\qquad$ If $\Delta^k > \nabla^k$, then go to Step 3.1.  \\ $\qquad$ If
$\Delta^k < \nabla^k$, then go to Step 3.2. 
\\ 
$\qquad$ If $\Delta^k = \nabla^k$, then set $x_j^* := l_j$ for $j \in
L(\hat{\mathbf{x}}^k)$, $x_j^* := u_j$ for $j \in U(\hat{\mathbf{x}}^k)$,
\\ $\qquad$ $\quad$ 
$x_j^* := x_j^k$ for $j \in J^k \setminus \{ L(\hat{\mathbf{x}}^k) \cup
U(\hat{\mathbf{x}}^k) \}$, and stop.  
\\ 
\textbf{Step 3.1} (peg lower bounds): \\ 
$\qquad$ Set $x_j^* := l_j$ for $j \in L$, $b^{k+1} := b^k - \beta^k_-$ and
$J^{k+1} := J^k \setminus L$.
 \\ $\qquad$ If
$J^{k+1} := \emptyset$ then stop, else set $k := k+1$ and go to Step 1. 
\\ 
\textbf{Step 3.2} (peg upper bounds): 
\\ $\qquad$ Set $x_j^* := u_j$ for
$j \in U$, $b^{k+1} := b^k - \beta^k_+$ and $J^{k+1} :=
J^k \setminus U$.  \\ $\qquad$ If $J^{k+1} :=
\emptyset$ then stop else set $k:=k+1$ and go to Step 1. \\

\\
\end{tabular}

\noindent We need to clarify some of the steps of the algorithm. In Step 1, we
find $\hat{\mathbf{x}}^{k+1}$ from, or partly from, $\hat{\mathbf{x}}^{k}$.
Assume, for example, that $\phi_j(x_j) = x_j\log(x_j/a_j-1)$ and
$g_j(x_j)=x_j$; then, $x_j^{k+1} = a_j b^{k+1} / \sum_{j\in J^{k+1}} a_j =
b^{k+1} / (\omega -\sum_{j\in K(\hat{\mathbf{x}}^k)} a_j$), where $\omega =
\sum_{j\in J^{k}} a_j$ and $K := U$ if the upper bound was pegged at iteration
$k$ and $K := L$ if the lower bound was pegged at iteration $k$. If $|K| < |
J^{k+1} |$ then this will save us some operations.  A similar update of
$\hat{\mathbf{x}}$ for the quadratic knapsack problem is performed in
\cite[Section 3]{RJL92}, \cite{BSS96} and \cite[Section 5.1]{Kiw08b}.

As in \cite[Algorithm 3.1]{Kiw08b}, our algorithm will stop if $\Delta^k =
\nabla^k$, while the algorithm in \cite[Section 2]{BrS02} stops only if
$L(\hat{x}^k) \cup U(\hat{x}^k) = \emptyset$. Moreover, in Steps 3.1 and 3.2,
we peg the variables that violate the bounds and calculate $b^k$ explicitly,
while in \cite[Section 2]{BrS02} the index $j$ is added to the set of violated
bounds ($L$ or $U$) and $b^k$ is calculated as $b - \sum_{j \in L} g_j(l_j)
- \sum_{j \in U} g_j(u_j)$.

According to Proposition \ref{pegglow}, if Case 2 in Section \ref{sec:taop}
holds, Step 2 in the algorithm is modified as follows (an analogous
modification can be defined for the relaxation algorithms in Sections
\ref{sec:mod2} through \ref{sec:tandf}):

\begin{tabular}{l}
\\ \textbf{Step 2'} (implicit evaluation): 
\\ 
$\qquad$ Determine
$U(\hat{\mathbf{x}}^k)$ and $L(\hat{\mathbf{x}}^k)$ while computing $\Delta^k
= \alpha^k_+ -\beta^k_+ $ and $\nabla^k = \beta^k_-- \alpha^k_-$. 
\\ 
$\qquad$
If $\Delta^k > \nabla^k$, then go to Step 3.2. \\ $\qquad$ If $\Delta^k <
\nabla^k$, then go to Step 3.1. 
\\ 
$\qquad$ If $\Delta^k = \nabla^k$, then set $x_j^* :=
l_j$ for $j \in L(\hat{\mathbf{x}}^k)$, $x_j^* := u_j$ for $j \in
U(\hat{\mathbf{x}}^k)$, 
\\ 
$\qquad$ $\quad$ $x_j^* := x_j^k$ for $j \in J^k
\setminus \{ L(\hat{\mathbf{x}}^k) \cup U(\hat{\mathbf{x}}^k) \}$, and stop.
\\ \\
\end{tabular}

\begin{remark}
For the equality problem (\ref{problemeq}), $\mu$ is unrestricted; hence the
algorithm for problem (\ref{problemeq}) will be similar to the above, except
we ignore Step 0. 
\end{remark}

\begin{remark}
In Step 1 we have to calculate $\hat{x}_j$ $| J^k|$ times. In Step 2 we need
to find $x_j$ which needs at most $2| J^k |$ comparisons. We also have to
calculate $\nabla^k$ and $\Delta^k$, which implies $2| L \cup U |$ operations.
\end{remark}

\subsubsection{Algorithm: Dual determination with Implicit evaluation of the 
Relaxed problem (DIR2)}\label{sec:mod2} 

Instead of evaluating the primal variables in Step 1 as in the algorithm PIR2
in Section \ref{sec:bih}, we can evaluate the dual variable $\mu$. Note that
if we evaluate the dual variable, then we have to determine the breakpoints in
the initialization. Our modification of the algorithm in Section \ref{sec:bih}
takes the following form (Steps 0 and 3 are same as in PIR2 and are therefore
not repeated):

\begin{tabular}{l}
\\ 
\textbf{Initialization: } 
\\
$\qquad$ Set $k:=1$, $J^k := J$, and $b^k := b$. 
\\
$\qquad$ Calculate breakpoints
$\mu^l:= ( \mu^l_j )_{j \in J^k}$, and $\mu^u:= ( \mu^u_j )_{j \in J^k} $ as in (\ref{xmu}). 
\\  
\textbf{Step 1} (solve relaxed dual problem):
\\
$\qquad$ Find the optimal dual variable $\hat{\mu}^k$ of the relaxed problem (\ref{problemrelax}). 
\\ 
\textbf{Step 2} (implicit evaluation): \\ $\qquad$
Determine $L(\hat{\mu}^k)$ and $U(\hat{\mu}^k)$ while
computing $\Delta^k := \alpha^k_+ -\beta^k_+ $ and $\nabla^k :=
\beta^k_-- \alpha^k_-$. \\
$\qquad$ If $\Delta^k > \nabla^k$, then go to Step 3.1.  
\\ 
$\qquad$ If $\Delta^k < \nabla^k$, then go to Step 3.2. 
\\ 
$\qquad$ If $\Delta^k = \nabla^k$, then set $x_j^* := l_j$ for $j \in L(\hat{\mu}^k)$, $x_j^* := u_j$ for $j \in U(\hat{\mu}^k)$ 
\\ 
$\qquad$ $\quad$ $x_j^* := x_j^k$ for $j \in J \setminus \{ L(\hat{\mu}^k) \cup
U(\hat{\mu}^k) \}$, and stop.   
\\
\end{tabular} 

\begin{remark}
In Step 1 we need to find the optimal dual solution $\mu^k$ to the relaxed
problem. In Step 2 we need at most $2|J^k|$ comparisons but we only need to
calculate $\hat{x}_j(\mu)$ for $j \in \{ L \cup U \}$.  For the evaluation we
need to calculate $\nabla^k$ and $\Delta^k$, which implies $2| L \cup U |$
operations.
\end{remark}

\subsubsection{Algorithm: Dual determination with Explicit evaluation of the 
Relaxed problem (DER2) }\label{sec:ste} As in the algorithm DIR2 in Section
\ref{sec:mod2} we evaluate the dual variable but instead of evaluating the
relaxed solution implicitly we do it explicitly. Define $\beta_l := \sum_{j\in
  L} g_j(l_j)$ and $\beta_u :=\sum_{j\in U} g_j(u_j)$. The algorithm follows
(Steps 0 and 3 are same as in PIR2 in Section \ref{sec:bih} and are therefore
not repeated):

\begin{tabular}{l}
\\ \textbf{Initialization: }
\\ 
$\qquad$ Set $k:=1$, $J^k := J$, and $b^k := b$. 
\\
$\qquad$ Calculate breakpoints
$\mu^l:= ( \mu^l_j )_{j \in J^k }$, $\mu^u:= ( \mu^u_j )_{j \in J^k }$ as in 
(\ref{xmu}).
\\  
\textbf{Step 1} (solve relaxed dual problem):
\\ 
$\qquad$ Find the optimal dual variable $\hat{\mu}^k$ of the relaxed problem 
(\ref{problemrelax}). 
\\ 
\textbf{Step 2} (explicit evaluation):
\\ 
$\qquad$ Determine $U(\hat{\mu}^k)$ and $L(\hat{\mu}^k)$ and
calculate $\delta (\hat\mu^k) := \sum_{j \in J^k \setminus \{
 U(\hat\mu^k) \cup L(\hat\mu^k) \} } g_j(x_j(\hat\mu^k)) + \beta_l + \beta_u$.
\\ 
 $\qquad$ If $\delta (\hat\mu^k) > b^k$, then go to Step 3.1. 
\\ 
$\qquad$ If $\delta (\hat\mu^k) < b^k$, then go to Step 3.2. 
 \\ 
 $\qquad$ If $\delta (\hat\mu^k) = b^k$, then set $x_j^* := l_j$ for $j \in 
L(\hat{\mu}^k)$, $x_j^* := u_j$ for $j \in U (\hat{\mu}^k)$ 
\\ 
$\qquad$ $\quad$ $x_j^* := x_j^k$ for $j \in J \setminus \{ L(\hat{\mu}^k) \cup
U(\hat{\mu}^k) \}$, and stop.
\\ \\
\end{tabular}

\noindent 
The algorithm uses the principle of the algorithm in \cite[Algorithm 1]{Ste01}.

\begin{remark}
In Step 1 we need to find the value of $\hat\mu^k$ from the relaxed
problem. In Step 2 we need at most $2|J^k|$ comparisons and we need to
calculate $x^k_j$ for $j \in J^k \setminus \{ L \cup U \}$. For the evaluation
we need to calculate $\delta(\hat\mu^k)$, which implies $|J^k|$ operations.
\end{remark}

\subsubsection{Algorithm: Dual determination modification with blended 
evaluation of the Relaxed problem (DBR2)} \label{sec:bh4} 

Consider the implicit evaluation in Section \ref{sec:mod2}. We calculate
$\nabla^k$ and $\Delta^k$ from $\nabla^k := \sum_{j\in L(\hat\mu^k)} g_j(l_j)
- \sum_{j\in U(\hat\mu^k)} g_j(\hat{x}^k_j)$ and $\Delta^k := \sum_{j\in
  U(\hat\mu^k)} g_j(\hat{x}^k_j) - \sum_{j\in L(\hat\mu^k)} g_j(u_j)$, which
implies $2P|L^k \cup U^k |$ operations, where $P$ is an integer associated
with the number of operations it takes to calculate $g_j(x_j)$. Moreover, we
have to determine $\hat{x}_j$ for $j \in \{ L^k \cup U^k \}$, which implies
$Q|L^k \cup U^k |$ operations, where $Q$ is an integer associated with the
number of operations it takes to determine $x_j(\hat\mu)$.

Now consider the explicit evaluation: We have to calculate $\delta (\hat\mu^k)
:= \sum_{j \in J^k \setminus \{ U(\hat\mu^k) \cup L(\hat\mu^k) \} }
g_j(x_j(\hat\mu^k)) + \sum_{j \in L(\hat\mu^k)} g_j(l_j) + \sum_{j \in
  U(\hat\mu^k)} g_j(u_j)$, which implies $P|J^k|$ operations. Moreover, we
have to determine $\hat{x}_j$ for $j \in J^k \setminus \{ L^k \cup U^k \}$,
which implies $Q|J^k \setminus \{ L^k \cup U^k \} |$ operations.

Hence, if $(P+Q)|J^k| < (2P+2Q)|U(\mu^k) \cup L(\mu^k)|$ or, equivalently,
$|J^k| < 2|U(\mu^k) \cup L(\mu^k)|$, then using the explicit evaluation of the
relaxed solution $\mathbf{\hat{x}}$ in Step 2 would require less operations,
and it would be more successful to use the algorithm in Section
\ref{sec:ste}. If however $|J^k| > 2|U(\mu^k) \cup L(\mu^k)|$, then there will
be less operations if we use the algorithm in Section \ref{sec:mod2}. So, we
propose a new algorithm that utilizes the cardinalities of the sets $J^k$,
$U(\hat{\mu}^k)$ and $L(\hat{\mu}^k)$: from the cardinalities we make the
decision whether to use an explicit or implicit evaluation in Step 2. We
consider the following modification of the algorithm PIR2 in Section
\ref{sec:bih}:

\begin{tabular}{l}
\\ \textbf{Initialization: }
\\ 
$\qquad$ Set $k:=1$, $J^k := J$, and $b^k := b$. 
\\
$\qquad$ Calculate breakpoints
$\mu^l:= ( \mu^l_j )_{j \in J^k }$, $\mu^u:= ( \mu^u_j )_{j \in J^k }$ as in
(\ref{xmu}).    
\\  
\textbf{Step 1: } (solve relaxed dual problem):
\\
$\qquad$ Find the optimal dual variable $\hat{\mu}^k$ of the relaxed problem
(\ref{problemrelax}). 
\\ 
\textbf{Step 1.1: } (implicit or explicit evaluation?):
\\ 
$\qquad$
Determine $U(\hat{\mu}^k)$ and $L(\hat{\mu}^k)$. 
\\ 
$\qquad$ If $|J^k| < 2|U(\hat{\mu}^k) \cup L(\hat{\mu}^k)|$ then use explicit
evaluation (continue with algorithm DER2 in 
\\ 
$\qquad$
Section  \ref{sec:ste}), otherwise use implicit evaluation (continue with
algorithm DIR2 in Section \ref{sec:mod2}).  
\\ \\
\end{tabular}

\subsubsection{3- and 5-sets pegging for relaxation algorithms}\label{sec:tandf}
From the proof of convergence of the relaxation algorithm (see Section
\ref{sec:proofrel}) we have that the algorithm improves the lower or the upper
bound for the dual variable in each iteration. Hence, similar to the
breakpoint algorithm (see Sections \ref{sec:rmm} and \ref{sec:MB5}) it is
possible to apply 3- and 5-sets pegging for the dual relaxation algorithms
DIR2, DER2 and DBR2. Similar to the relaxation algorithm we will denote a
relaxation algorithm that uses 3-sets-pegging with a suffix ''3'', e.g., DBR3,
and one that uses 5-sets-pegging with a suffix ''5'', e.g., DBR5. The
implementation of 3- and 5-sets pegging is similar for the three dual
relaxation algorithms; therefore only DIR3 and DIR5 are given (Step 0 is
similar as for PIR in Section \ref{sec:bih}):

\begin{tabular}{l}
\\ 
\textbf{DIR3: } \\
\textbf{Initialization: } 
\\
$\qquad$ Set $k:=1$, $J^k := J$, $M^k := \emptyset $, $b^k := b$, $\underline{\mu} = -\infty$, and $\overline{\mu} = \infty$. 
\\
$\qquad$ Calculate breakpoints
$\mu^l:= ( \mu^l_j )_{j \in J^k }$, $\mu^u:= ( \mu^u_j )_{j \in J^k}$ as in (\ref{xmu}).
\\ 
\textbf{Step 1} (solve relaxed dual problem):
\\ 
$\qquad$ Find the optimal dual variable $\hat{\mu}^k$ of the relaxed problem (\ref{problemrelax}). 
\\ 
\textbf{Step 2} (implicit evaluation): \\ $\qquad$
Determine $L(\hat{\mu}^k)$ and $U(\hat{\mu}^k)$ from $J^k$ and
compute $\Delta^k = \alpha^k_+ -\beta^k_+ $ and $\nabla^k =
\beta^k_-- \alpha^k_-$. \\
$\qquad$ If $\Delta^k > \nabla^k$, then go to Step 3.1.  
\\ 
$\qquad$ If $\Delta^k < \nabla^k$, then go to Step 3.2. 
\\ 
$\qquad$ If $\Delta^k = \nabla^k$, then set $x_j^* = l_j$ for $j \in L(\hat{\mu}^k)$, $x_j^* = u_j$ for $j \in U(\hat{\mu}^k)$ 
\\ 
$\qquad$ $\quad$ $x_j^* = x_j(\hat{\mu}^k)$ for $j \in \{ J^k \cup M^k \} \setminus \{ L(\hat{\mu}^k) \cup
U(\hat{\mu}^k) \}$, and stop.   
\\ 
\textbf{Step 3.1} (peg lower bounds): 
\\ 
$\qquad$ Update lower bound, $\underline{\mu} := \hat{\mu}^k$, and resource $b^{k+1} := b^k - \beta^k_-$. \\  
$\qquad$ Set $x_j^* := l_j$ for $j \in L^k$, $J^{k+1} := J^k \setminus L^k$ and $M^{k+1} := M^{k}$. 
\\ 
$\qquad$ For $j \in J^{k+1}$: If $\underline{\mu},\overline{\mu} \in [\mu_j^u,\mu_j^l]$
then $J^{k+1} := J^{k+1}
\setminus \{j\}$, $M^{k+1} := M^{k+1} \cup \{j\}$.
\\
$\qquad$  If $J^{k+1} = \emptyset$ then find optimal solution and stop, else set $k := k+1$ and go to Step 1. 
\\ 
\textbf{Step 3.2} (peg upper bounds): 
\\
$\qquad$ Update upper bound, $\bar{\mu} := \hat{\mu}^k$, and resource $b^{k+1} := b^k - \beta^k_+$. 
\\  
$\qquad$ Set $x_j^* := u_j$ for $j \in U^k$, $J^{k+1} := J^k \setminus U^k$ and $M^{k+1} := M^{k}$. 
\\ 
$\qquad$ For $j \in J^{k+1}$: If $\underline{\mu},\overline{\mu} \in [\mu_j^u,\mu_j^l]$
then $J^{k+1} := J^{k+1}
\setminus \{j\}$, $M^{k+1} := M^{k+1} \cup \{j\}$.
\\
$\qquad$ If $J^{k+1} = \emptyset$ then find optimal solution and stop, else set $k := k+1$ and go to Step 1. 
\\
\\
\end{tabular}

\noindent In Steps 3.1 and 3.2, if $J^{k+1} = \emptyset$ then we can find the
optimal solution from $M$ since we know that for all $j \in M$ it holds that
$l_j < x_j < u_j$, i.e., we don't have to consider the constraint
(\ref{problemc}). Further, in Step 3.1 (and analogusly for Step 3.2) when
searching for $j\in J^{k+1}$ such that $\underline{\mu},\overline{\mu} \in
[\mu_j^u,\mu_j^l]$, we only have to consider $j\in J^{k+1} \setminus U^k
$, since we know that if $j \in U^k$ then $\hat{\mu}^k \le \mu_j^u$. Note that
for the case where $\hat{\mu}^k = \mu_j^u$ the corresponding index $j$ will
continue to belong to $J^{k+1}$. Finally we note that we don't have to check
if $\underline{\mu},\overline{\mu} \in [\mu_j^u,\mu_j^l]$ if $\underline{\mu}
= -\infty$ and/or $\overline{\mu} = \infty$.

\begin{tabular}{l}
\\ 
\textbf{DIR5: } \\

\textbf{Initialization: } 
\\
$\qquad$ Set $k:=1$, $J^k := J$, $M^k = U^k_+ = L^k_- := \emptyset $, $b^k := b$, $\underline{\mu} = -\infty$, and $\overline{\mu} = \infty$. 
\\
$\qquad$ Calculate breakpoints
$\mu^l:= ( \mu^l_j )_{j \in J^k }$, $\mu^u:= ( \mu^u_j )_{j \in J^k }$ as in (\ref{xmu}).
\\ 
\textbf{Step 1} (solve relaxed dual problem):
\\ 
$\qquad$ Find the optimal dual variable $\hat{\mu}^k$ of the relaxed problem (\ref{problemrelax}). 
\\ 
\textbf{Step 2} (implicit evaluation): \\ $\qquad$
Determine $L(\hat{\mu}^k)$ and $U(\hat{\mu}^k)$ from $J^k  \cup L_-^{k} \cup U_+^{k}$,
compute $\Delta^k = \alpha^k_+ -\beta^k_+ $ and $\nabla^k =
\beta^k_-- \alpha^k_-$. \\
$\qquad$ If $\Delta^k > \nabla^k$, then go to Step 3.1.  
\\ 
$\qquad$ If $\Delta^k < \nabla^k$, then go to Step 3.2. 
\\ 
$\qquad$ If $\Delta^k = \nabla^k$, then set $x_j^* = l_j$ for $j \in L(\hat{\mu}^k)$, $x_j^* = u_j$ for $j \in U(\hat{\mu}^k)$ 
\\ 
$\qquad$ $\quad$ $x_j^* = x_j(\hat{\mu}^k)$ for $j \in \{ J^k \cup M^k \cup L^k_- \cup U^k_+  \} \setminus \{ L(\hat{\mu}^k) \cup U(\hat{\mu}^k) \}$, and stop.   
\\ 
\textbf{Step 3.1} (peg lower bounds): 
\\ 
$\qquad$ Update lower bound, $\underline{\mu} := \hat{\mu}^k$, and resource $b^{k+1} := b^k - \beta^k_-$. 
\\
$\qquad$ Set $x_j^* := l_j$ for $j \in L^k$, $J^{k+1} := J^k \setminus L^k$ and $L_-^{k+1} := L_-^{k} \setminus L^k$. 
\\
$\qquad$ For $j \in \{ J^{k+1} \setminus U^k_+ \}$: If $\mu_m \le \mu_j^l$ then $J^{k+1} := J^{k+1} \setminus \{j\}$ and $L_-^{k+1} := L_-^{k+1}  \cup \{j\}$. 
\\
$\qquad$ For $j \in \{ U_+^k \setminus U^k \}$: If $\underline{\mu}\le\mu^l_j$
then  $U_+^k  := U_+^k \setminus \{j\}$ and $M^k := M^k \cup \{j\}$. 
\\ 
$\qquad$ Set $M^{k+1} := M^{k}$ and $U_+^{k+1} := U_+^{k}$. \\ 
$\qquad$ If $J^{k+1} \cup L_-^{k+1} \cup U_+^{k+1}= \emptyset$ then find optimal solution and stop,
\\
$\qquad$ $\quad$ else set $k := k+1$ and  go to Step 1.
\\ 
\textbf{Step 3.2} (peg upper bounds):
\\ 
$\qquad$ Update upper bound, $\bar{\mu} := \hat{\mu}^k$, and resource $b^{k+1} := b^k - \beta^k_+$. 
\\
$\qquad$ Set $x_j^* := u_j$ for $j \in U^k$, $J^{k+1} := J^k \setminus U^k$ and $U_+^{k+1} := U_+^{k} \setminus U^k$. 
\\
$\qquad$ For $j \in \{ J^{k+1} \setminus L^k_- \}$: If $\mu_m \le \mu_j^l$ then $J^{k+1} := J^{k+1} \setminus \{j\}$ and $U_+^{k+1} := U_+^{k+1}  \cup \{j\}$. 
\\
$\qquad$ For $j \in \{ L_-^k \setminus L^k \}$: If $\underline{\mu}\le\mu^l_j$
then  $L_-^k  := L_-^k \setminus \{j\}$ and $M^k := M^k \cup \{j\}$. 
\\ 
$\qquad$ Set $M^{k+1} := M^{k}$ and $L_-^{k+1} := L_-^{k}$. \\ 
$\qquad$ If $J^{k+1} \cup L_-^{k+1} \cup U_+^{k+1}= \emptyset$ then find optimal solution and stop, 
\\
$\qquad$ $\quad$ else set $k := k+1$ and  go to Step 1.
\\ 
\\
\end{tabular} 

\begin{remark}
Note that for DBR2 in Section \ref{sec:bh4} we determine whether we should
continue with DIR or DER from $| J^k | < 2 | U(\hat{\mu}^k) \cup
L(\hat{\mu}^k) |$. This condition needs to be modified for the 3- and the
5-sets algorithms DBR3 and DBR5. This is done by substituting the
condition mentioned by $| J^k \cup M^k | < 2 | U^k(\hat{\mu}^k) \cup L^k(\hat{\mu}^k) |$
and $| J^k \cup M^k \cup L_-^k \cup U_+^k | < 2 | U^k(\hat{\mu}^k) \cup
L^k(\hat{\mu}^k) |$ for 3- and 5-sets pegging, respectively.
\end{remark}

\subsection{Optimality of the relaxation algorithms} \label{sec:proofrel}

For the inequality problem (\ref{problem}), optimality and validation for the
pegging process for the primal relaxation algorithm PIR2 was established by
Bretthauer and Shetty \cite[Propositions 1--9]{BrS02}.  Let $k^*$ be the
iteration where the algorithm terminates. Then Bretthauer and Shetty state
that the primal relaxation algorithm generates the following solution for
problem (\ref{problemeq}):

\begin{subequations}\label{aopt}
\begin{align}
\mu &= \mu^{k^*} = -\phi_j^\prime(x_j^{k^*})/ g_j(x_j^{k^*}), & j &\in
J^{k^*}, \label{aopt1}  
\\ 
\rho_j &= \phi_j'(l_j^{k^*}) + \mu^{k^*} g_j'(x_j^{k^*}), & j &\in L, \label{aopt2} 
\\ 
\rho_j &= 0, & j &\in J^{k^*} \cup U, \label{aopt3} 
\\ 
\lambda_j &= -\phi_j'(u_j^{k^*}) - \mu^{k^*} g_j'(x_j^{k^*}), & j &\in
U, \label{aopt4} 
\\ 
\lambda_j &= 0, & j &\in J^{k^*} \cup L, \label{aopt5} 
\\ 
x_j &= l_j, & j &\in L, \label{aopt6} 
\\ 
x_j &= u_j, & j &\in U, \label{aopt7} 
\\ 
x_j &= x_j^{k^*}, & j &\in J^{k^*}. \label{aopt8}
\end{align}
\end{subequations} 

\noindent Bretthauer and Shetty establish that this solution fulfills all KKT
conditions, and therefore is optimal.

The optimality and convergence for the equality problem (\ref{problemeq}) can
be established similarly to the proof for the inequality problem given in
\cite{BrS02} except for the proof for feasibility of the dual variables
corresponding to the lower and upper bounds (Propositions 8 and 9 in
\cite{BrS02}). Additionally we do not have to prove that $\mu^* \ge 0$
(Proposition 4 in \cite{BrS02}) for the equality problem. Therefore we give a
complementary proof for the feasibility of the dual variables corresponding to
the lower and upper bounds. For the equality problem (\ref{problemeq}) we have
that $g_j(x_j) = a_jx_j$ and in the proof we will assume that $a_j > 0$.

\begin{lemma}\label{nd3}
Consider the equality problem (\ref{problemeq}) and algorithm PIR2 in Section
\ref{sec:bih}.

(a) If $\nabla^k > \Delta^k$ then $\mu^{k^*}  \ge \mu^k$.

(b) If $\nabla^k < \Delta^k$ then $\mu^{k^*}  \le \mu^k$.
\end{lemma}

\begin{proof}
The proof is similar to that of Proposition 7 in \cite{BrS02}.
\end{proof}

\begin{proposition}[feasibility of dual variables corresponding to 
bounds]\label{mythm} For problem (\ref{problemeq}), the solution (\ref{aopt})
  generated by PIR2 in Section \ref{sec:bih} satisfies the feasibility of the
  dual variables ($\rho_j$ and $\lambda_j$) corresponding to the lower and
  upper bounds.
\end{proposition}

\begin{proof}
(i) For $j \in J^{k^*} \cup U$, we have from (\ref{aopt3})
that $\rho_j = 0$.

(ii) For $j \in L$, we have from (\ref{aopt3}) that $\rho_j = \phi_j'(l_j) +
\mu^{k^*} a_j$. We know that all variables $x_j$ with $j \in L$ were pegged in
the iterations $k$ where $\nabla^k > \Delta^k$. For these iterations $k$ we
have $\hat{x}_j^k \le l_j$. Further, from the convexity of $\phi_j$ and the
assumption that $a_j>0$, we have that $\mu (x_j) = -\phi_j'(x_j)/a_j$ is
decreasing in $x_j$. Hence,
\begin{equation}
\frac{\rho_j}{a_j} = \frac{\phi_j'(l_j)}{a_j} + \mu^{k^*} \ge
\frac{\phi_j'(\hat{x}_j^k)}{a_j} + \mu^{k^*} = -\mu^k + \mu^{k^*} \ge 0,
\end{equation} 
where the last inequality follows from Lemma \ref{nd3}(a).

(iii) For $j \in J^{k^*} \cup L$, we have from
(\ref{aopt3}) that $\lambda_j = 0$.

(iv) For $j \in U$, we have from (\ref{aopt5}) that $\lambda_j = -
\phi_j'(u_j) - \mu^{k^*} a_ju_j$. We know that all variables $x_j$ with $j \in
U$ were pegged in the iterations where $\nabla^k < \Delta^k$. For these
iterations $k$ we have $\hat{x}_j^k \ge u_j$. Further, from the convexity of
$\phi_j$ and the assumption that $a_j>0$, we have that $\mu (x_j) =
-\phi_j'(x_j)/a_j$ is decreasing in $x_j$. Hence,
\begin{equation}
\frac{\lambda_j}{a_j} = -\frac{\phi_j'(u_j)}{g'_j(u_j)} - \mu^{k^*} \ge
-\frac{\phi_j'(\hat{x}_j^k)}{g'_j(\hat{x}_j^k)} - \mu^{k^*} = \mu^k -
\mu^{k^*} \ge 0,
\end{equation} 
where the last inequality follows from Lemma \ref{nd3}(b). 
\end{proof}

\noindent The algorithms in Sections \ref{sec:mod2}, \ref{sec:ste}, and
\ref{sec:bh4} will also converge to the optimal solution since they are
equivalent to PIR.

\begin{remark}
If we introduce the additional assumption that $\phi_j$ and $g_j$ are twice
differentiable and that $g_j^\prime >0$, an alternative convergence result for
the dual relaxation algorithm is found in \cite{Ste01}.
\end{remark}

\subsection{Time complexity}
Consider algorithm PIR2 in Section \ref{sec:bih}. In Step 0 we need at most
$2n$ comparisons to determine the primal variables and $C_0n$ operations, for
a constant $C_0$, to determine if the solution is feasible or not. In Step 1,
we solve the relaxed problem, which gives $C_1n$ operations for a constant
$C_1$. In Step 2 we perform at most $2n$ comparisons to determine the lower
and upper sets and we compute $\Delta^k$ and $\nabla^k$, which gives at most
$C_2n$ operations for a constant $C_2$. Steps 3.1 and 3.2 give at most $2n+1$
operations. Further, in the worst case the algorithm only pegs one primal
variable in each iteration, which results in $n$ iterations. Hence, we conclude
that the algorithm has a complexity of $O(n^2)$.


\section{A quasi-Newton algorithm (NZ)} \label{sec:zen_algor}

Nielsen and Zenios \cite[Section 1.4]{NiZ92} develop a quasi-Newton method for
finding the dual optimal solution $\mu^*$ of the problem (\ref{problemeq}). It
is assumed that the objective term $\phi_j$ is strictly convex with a
derivative $\phi^\prime_j$ whose range is $\mathbb{R}$. They compare their
numerical method with three linesearch methods \cite{HKL80,CeL81,Tse90}. Their
results show that their numerical method always performs well compared with
the other algorithms. They implement their algorithms on a massively parallel
computer. This is not our intention. But since the algorithm seems to perform
well on parallel computers it makes sense to evaluate it on non-parallel
computers.

Let $f_j$, $j \in J$, be the inverse of $\phi^\prime_j$ such that
\begin{equation}
  \text{for } j \in J, \quad \phi_j(f_j(\mu)) = \mu, \quad \mu \in \mathbb{R}.
\end{equation}  
Similar to (\ref{xmu}) we conclude that
\begin{equation}
 x_j(\mu) = \max{\{l_j,\min{\{f_j(a_j\mu),u_j\}} \}} \label{min_x}.
\end{equation}
The heart of the algorithm is, like in the breakpoint algorithm, to find $\mu$
such that the primal constraint (\ref{problemb}) is fulfilled. In other words
find $\mu$ such that
\begin{equation}
\Psi(\mu) := b - \sum_{j \in J} a_jx_j(\mu) = 0.
\label{psi}
\end{equation}
Nielsen and Zenios \cite[Section 1.4]{NiZ92} define two functions $\Phi_j^+$
and $\Phi_j^-$:
\begin{equation}
  \Phi_j^+(\mu) := \left\{ 
  \begin{array}{l l l }
      \min{\{f_j(a_j\mu),u_j\}}, & \quad \mathrm{if }\;a_j>0,\\
     \max{\{l_j,f_j(a_j\mu)\}}, & \quad \mathrm{if }\; a_j<0,
  \end{array} \right.
  j \in J,
 \end{equation}
and
\begin{equation}
  \Phi_j^-(\mu) := \left\{ 
  \begin{array}{l l l }
      \min{\{f_j(a_j\mu),u_j\}}, & \quad \mathrm{if }\;a_j<0,\\
     \max{\{l_j,f_j(a_j\mu)\}}, & \quad \mathrm{if }\; a_j>0.
  \end{array} \right.
    j \in J,
 \end{equation}
Note that for $j \in J$, if $a_j > 0$ and $f_j$ is concave and increasing
then $\Phi^+_j$ is concave and if $a_j > 0$ and $f_j$ is convex and decreasing
then $\Phi^-_j$ is convex. Further, we define two sets of indices such that
$J^+ := \{ \, j \in J \mid a_j>0 \, \}$ and $J^- :=\{ \, j \in J \mid a_j<0 \,
\} $.  Define two approximation of $\Psi$ such that
\begin{subequations}
\begin{align}
	\Psi^+(\mu) & := b - \sum_{j \in J} a_j \Phi^+(\mu) \nonumber \\ & =
        b- \sum_{j\in J^+} a_j \min{\{g(a_j\mu),u_j\}} - \sum _{j\in J^-}
        \max{\{l_j,g(a_j\mu)\}},
\end{align}
\end{subequations}
and
\begin{subequations}\label{rox1}
\begin{align}
	\Psi^-(\mu) & := b - \sum_{j \in J} a_j \Phi^-(\mu) \nonumber \\ & = b-
        \sum_{j\in J^-} a_j \min{\{g(a_j\mu),u_j\}} - \sum _{j\in
          J^+} \max{\{l_j,g(a_j\mu)\}}.
\end{align}
\end{subequations}
Note that if $a_j>0$ and $f_j$ is concave then $\Psi^+$ is convex
and if $a_j>0$ and $f_j$ is convex then $\Psi^-$ is concave. Define the sub- and
superdifferentials of $\Psi^+$, respectively, $\Psi^-$, as
\begin{subequations}\label{rox2}
\begin{align}
	\partial \Psi^+(\mu) &:= \{ \, d \in \mathbb{R} \mid
        (\Psi^+(\mu^\prime ) - \Psi^+(\mu) \ge d(\mu^\prime - \mu) \quad
        \forall \mu^\prime \in \mathbb{R} \, \}, \\ \partial \Psi^-(\mu) &:=
        \{ \, d \in \mathbb{R} \mid (\Psi^-(\mu^\prime) - \Psi^-(\mu) \le
        d(\mu^\prime - \mu) \quad \forall \mu^\prime \in \mathbb{R} \, \}.
\end{align}
\end{subequations}
Further, define $\mu^*_\varepsilon$ and $\mathbf{x}_\varepsilon^*$ as the approximate
dual and primal solution such that $|\Psi(\mu^*_\varepsilon)| < \varepsilon$ where
$\varepsilon > 0$. The algorithm (NZ) follows (\cite[Linesearch 4]{NiZ92}):

\begin{tabular}{l}
\\ \textbf{Initialization: } 
Set $\varepsilon >0$, $k = 0$, $\mu^0 \in \mathbb{R}$. 
\\ 
\textbf{Iterative algorithm:} 
\\ 
\textbf{Step 1} (compute step size):
 \\
 $\quad$ If $\Psi(\mu^k) > \varepsilon$ then 
 \\ 
 $\qquad \Delta \mu ^{k+1} := -\frac{\Psi(\mu^k)}{d^k}$ where $d^k \in \partial
\Psi^+(\mu^k)$; go to Step 2, 
\\ 
$\quad$ else if $\Psi(\mu^k) < -\varepsilon$ then
\\ 
$\qquad \Delta \mu ^{k+1} := -\frac{\Psi(\mu^k)}{d^k}$ where $d^k
\in \partial \Psi^-(\mu^k)$; go to Step 2,
\\ 
$\quad$ else \\ $\qquad$ Set
$\mu^*_\varepsilon := \mu^k$ and determine $\mathbf{x}_\varepsilon^*$ from
(\ref{min_x}). Stop.
\\ 
\textbf{Step 2} (dual variable update):
 \\ 
 $\quad$ set
$\mu^{k+1} := \mu^k + \Delta \mu ^{k+1} $. 
\\ 
\textbf{Step 3: } Set $k := k +1$ and go to
Step 1. \\ \\
\end{tabular}

\noindent The algorithm converges to a value $\mu^*_\varepsilon$, such that
$|\Psi(\mu^*_\varepsilon)| < \varepsilon$ if the objective function components
$\phi_j$ is such that the corresponding function $\Psi^+(\mu)$ is convex or if
the corresponding function $\Psi^-(\mu)$ is concave \cite[Proposition
  8]{NiZ92}. For some problems, the inverse of the derivative might however
result in imaginary values. One solution to this problem is to consider the
equivalent maximization problem of (\ref{problemeq}), i.e., to $\maximize_x
-\phi(x)$.

\begin{remark}
In practice we will choose $\varepsilon > 0$ and the algorithm will in most
cases stop such that $| \Psi(\mu^k) | > 0$. Hence we will end up with an
approximate solution of the optimal dual variable $\mu^*$.  The map from the
dual space to the primal might not be linear. Hence the primal error might be
larger than we expect, i.e., $|x _\varepsilon^* - x^*| >> | \mu _\varepsilon^*
- \mu^* | $. However, there are methods for generating primal optimal
solutions from any Lagrangian dual vector (see for example
\cite{LMOP07}). Another plausible method to find the optimal solution from the
approximate solution $\mu_{\varepsilon}^*$ is to use a breakpoint or a
relaxation algorithm, starting from $\mu_{\varepsilon}^*$.
\end{remark}

\section{Method of algorithm evaluation}\label{sec:method}

This section serves to provide an overview of the procedure for the numerical
study.  In Section \ref{sec:limitations} we define problem instances for the
numerical study.  Some theory on how the problem instances can be designed
follows in Section \ref{sec:dop}. In Section \ref{sec:performance} we give a
brief overview of performance profiles (\cite{DoM02}) which are used for the
evaluation of the numerical study. Finally, in Section \ref{sec:plc}, we
describe the computational environment.

\subsection{Problem set}\label{sec:limitations}

For the numerical study we consider five common special cases of problem
(\ref{problemeq}); it also covers the inequality problem (\ref{problem}) when
$\mu^* > 0$. Only finite values of the lower an upper bounds
(\ref{problemeqc}) are considered, i.e., for $j \in J$, $l_j > -\infty$ and
$u_j < \infty$. The five problem cases are briefly specified next:

\paragraph{Quadratic problem}
The convex separable quadratic problem is the special case of
(\ref{problemeq}), where
\begin{equation}\label{quadratic}
\phi _j ( x_j) = \frac{w_j}{2}x^2_j - c_jx_j \; \text{ for } j \in J,
\end{equation}
where $w_j, c_j > 0$, $j \in J$. Numerical studies of algorithms for problem
(\ref{quadratic}) are widely explored; e.g., see
\cite{NiZ92,Kiw07,Kiw08a,Kiw08b}.  In our numerical study the parameters are
randomized such that $a_j \in [1,30]$, $w_j \in [1,20]$, $c_j \in [1,25]$,
$l_j \in [0,3]$ and $u_j \in (3,11]$.


\paragraph{Stratified sampling}

If we have a large population and would like to perform a statistical research
among the population, it is practically infeasible to examine every single
individual in the population. Instead we can stratify the population into $n$
strata. An example of a stratum might be people of a certain age. Let $M$ be
the number of individuals in the entire population and $M_j$ the number of
individuals in strata $j$. If we want to minimize the variance of the entire
population we need to allocate the number of samples $x_j$ from each strata
from:
\begin{equation} \label{stratified}
\phi_j(x_j) = \omega_j \frac{(M-x_j)\rho ^2}{(M-1)x_j} \; \text{ for } j \in J,
\end{equation}
where $\omega_j = M_j/M$ and $\rho_j$ is an estimate of the variance for
strata $j$. From $b$ in the resource constraint we specify the total sample
size. In our numerical study the parameters are
randomized such that $a_j \in [1,30]$, $m_j \in [5,30]$, $c_j \in [1,4]$, $l_j
\in [1,3]$ and $u_j \in (3,15]$.


\paragraph{Sampling}
According to \cite{BRS99}, often the objective terms for sampling problems can
be written as
\begin{equation}\label{production}
\phi_j (x_j) =  c_j/x_j \; \text{ for } j \in J.    
\end{equation}
In our numerical study the parameters are randomized such that $a_j \in
[1,4]$, $c_j \in [5,30]$, $l_j \in [0,3]$ and $u_j \in (3,6]$.

\paragraph{The theory of search}
In the theory of search problem it is determined how a resource $b$ of time
should be spent to find an object among $n$ subdivisions of an area with the
largest probability. It is assumed that we know the probability $m_j$ for an
object to be found in subdivision $j$. The objective component $\phi_j$
describes the probability of finding the object in subdivision $j$ and takes
the form:
\begin{equation}\label{TEO}
\phi_j (x_j) = m_j(e^{-b_jx_j}-1) \; \text{ for } j \in J.
\end{equation}
This problem is widely explored by Koopman \cite{Koo53,Koo99} and the problem
is possible to apply to a large variation of search problems, e.g., searching
for refugees fleeing from Cuba \cite{Sto81}. In our numerical study the
parameters are randomized in the following intervals: $m_j \in [0.5, -8]$,
$b_j \in [0.1, 3]$, $a_j \in [1, 3]$, $l_j \in [0, 0.1]$ and $u_j \in (0.1,
5]$.

\paragraph{Negative entropy function}
The negative entropy function mentioned in \cite{NiZ92}:
\begin{equation} \label{negentropy}
\phi_j(x_j) = x_j\log{(\frac{x_j}{a_j} - 1)} \; \text{ for } j \in J.
\end{equation}
In our numerical study the parameters are
randomized such that $c_j \in [50,250]$, $l_j \in [20,100]$ and $u_j \in
(30,210]$.

\subsection{Design of problem instances} \label{sec:dop}
Similarly to the numerical study in \cite{KoL98}, we divide our set of test
problem instances into groups containing different portions of the activities
within the lower and upper bounds at the optimal solution. Let $H := \{ \, j
\in J \mid l_j < x^*_j < u_j \, \}$. Then, the percentage is determined as
$|H| / n$. To motivate this approach, we refer to the variance of CPU times for
different portions of active activities in \cite{KoL98}.  Further, similar to
the numerical study in \cite{Kiw07}, we consider different problem sizes $n$
since the theoretical CPU time for different algorithms vary between $O(n)$
and $O(n^2)$.

For the problem set, we pseudo randomize the parameters $l_j$, $u_j$ and all
the parameters associated with $\phi_j$ and $g_j$. In the numerical study we
use a linear resource constraint such that $g_j(x_j) = a_jx_j$, where $a_j>0$
for all $j \in J$. This simplifies the design of the problem set, since
$\phi_j$ is convex and $l_j < u_j$ for all $j \in J$. We have that
\begin{equation}\label{designprob2}
   x_j^* = \left\{ 
  \begin{array}{l l l }
    x_j^*, & \quad \text{if } \mu^* = -\phi_j^\prime(x_j^\prime)/ a_j, \\ l_j,
    & \quad \text{if } \mu^* \ge -\phi_j^\prime(l_j) / a_j \ge
    -\phi_j^\prime(u_j) / a_j, \\ u_j, & \quad \text{if } \mu^* \le
    -\phi_j^\prime(u_j) / a_j ) \le -\phi_j^\prime(l_j) / a_j.
  \end{array} \right.
\end{equation}

\noindent By using the properties of (\ref{designprob2}) we can determine $H$ 
such that $|H| / n = y$ for any $y \in [0,1] $.

 

\subsection{Performance profiles} \label{sec:performance}

Dolan and Mor\'{e} \cite{DoM02} propose a performance profile for the
evaluation of optimization software. The method is briefly summarized as
follows: Assume that we have a set $A$ of algorithms that consist of $n_{a}$
algorithms and a problem set $P$ that consists of $n_p$ problem instances. Let
$t_{p,a}$ denote the time it takes for algorithm $a\in A$ to solve problem
$p\in P$. A performance ratio $r_{p,a}$ is introduced:
\begin{equation}
  r_{p,a}(t_{p,a}) := \frac{t_{p,a}}{\min \{t_{p,l} \mid l\in A \}}.
\end{equation}
For a problem $p$, the performance ratio is a measure of how fast algorithm
$a$ is relative to the fastest algorithm solving problem $p$. Fix a constant
$r_M$ such that $r_M \ge r_{p,a}$ for all $p\in P$, $a \in A$
and let $r_{p,a} = r_M$ if algorithm $a$ fails to solve problem $p$. Further,
we introduce the distribution
\begin{equation}
  \rho_a (\tau) = \frac{1}{n_p} | \{p\in P \mid r_{p,a} \le \tau \} |,
\end{equation}
for each algorithm, where $|\cdot|$ denote the cardinality of a set and $\tau
\in [1, r_M]$. For algorithm $a$, the distribution $\rho_a$ describes the
percentage of problem instances that are solved at least as fast as $\tau$
times the fastest algorithm for problem $p$. Note that $\rho_a(1)$ is the
percentage for algorithm $a$ being the fastest. Moreover, $\lim_{\tau
  \rightarrow r_M} \rho_s(\tau)$ is the probability that algorithm $a$ will
solve a problem $p$ in $P$. If we have a large problem set $P$ then $\rho_a
(\tau )$ will not be affected much by a small change in $P$ (\cite[Theorem
  1]{DoM02}).

\subsection{Program language, computer and code}\label{sec:plc}
The algorithms are implemented in Fortran 95, compiled with {\em gfortran}
under mac OS X 10.8.2 (2.5 GHz Intel Core i5, 4 GB 1600 MHz DDR3).

\section{Computational experiments} \label{sec:compexp}

In this section we present the results from numerical experiments of the
problems defined in Section \ref{sec:limitations}.  If nothing else is
mentioned, 100 problem instances of each problem is evaluated for each problem
size. Further, the problem instances are designed such that different values
of $|H| / n$ are considered, see Section \ref{sec:dop}. 

The development of numerical experiments for the problem (1) is illustrated in
Table \ref{decade}, where we cite the size of the largest test problem
reported during each decade; the table is an extension of Table 2 in
\cite{Pat08}.

\begin{table}[h!]\
\centering
\caption{Largest problem instances solved for each algorithm class through the
  decades.}
\begin{tabular}{@{} l |  r  r  r  r  r  r r @{}} \label{decade}
Decade & 50s & 60s & 70s & 80s & 90s & 00s & 10s \\ [1ex] \hline Breakpoint
algorithm & 2 & 60 & 12 & 200 & $10^4$ & $2 \cdot 10^6$ & $30 \cdot 10^6$
\\[1ex] Relaxation algorithm & -- & -- & 200 & 200 & $10^4$ & $2
\cdot 10^6$ & $110 \cdot 10^6$ \\[1ex]
\end{tabular}
\end{table}
 
In Section \ref{sec:Bhprim} we present performance profiles for the relaxation
algorithms defined in Section \ref{sec:relaxation}.  In Section
\ref{sec:resrank} we evaluate the pegging process for both the breakpoint and
the relaxation algorithm defined in Section \ref{sec:rank}.  In Section
\ref{sec:final}, the best performing relaxation algorithm and breakpoint
algorithm are compared with the quasi-Newton algorithm in Section
\ref{sec:zen_algor}. In Section \ref{sec:acr} we give a critical review of the
numerical study. Finally, in Section \ref{sec:conclusion}, we make some
overall conclusions.

In the following Figures~\ref{bhvsdualbh}--\ref{finall} we show performance
profiles, as defined in Section~\ref{sec:performance}, for three competing
sets of algorithms. To summarize the appearance of these plots, for each plot
and corresponding competing set of algorithms, the graph for one algorithm
shows the portion (between $0$ and $1$ on the y-axis) of the problem instances
considered that are solved within $\tau$ (on the x-axis) times the fastest
algorithm in the test for problem $p$. In particular, the value of $\rho_a(1)$
is the portion of the problems in which algorithm $a$ is the fastest, and
$\lim_{\tau \rightarrow r_M} \rho_a(\tau)$ is the probability that algorithm
$a$ will solve a problem $p$ in $P$. (The latter information is particularly
illustrative and relevant for Figure~\ref{finall}.)

\subsection{Evaluation of the relaxation algorithm}\label{sec:Bhprim}

We evaluate the different paths in Figure \ref{relaxationOverview}.  First,
recall that PIR2 determines the primal variables while DIR2 determines the
dual variable. As can be seen in the leftmost performance profile in
Figure~\ref{bhvsdualbh}, DIR2 is the fastest in $67.8\%$ of the problem
instances solved. The result is what we can expect from theory, since DIR2
does not need as many operations in Step 1; see Sections \ref{sec:bih} and
\ref{sec:mod2}.  PIR2 is faster in $32.2\%$ of the problems solved; the latter
cases stem mainly from to the negative entropy problem (\ref{negentropy}) but
also from the theory of search problem (\ref{TEO}) and the stratified sampling
problem (\ref{stratified}) when $n$ and $| H | /n$ is small.  The reason for
PIR2 to be faster in almost all cases for the negative entropy problem is due
to the computationally simple expression of the primal variables $x_j^k(\mu^k)
= \frac{a_j}{\sum_{j\in J^k} c_j}$ while the dual variable is evaluated from
$\mu^k = \log{ \sum_{j \in J^k} c_j } - \log{b}$. Also the computations of the
breakpoints might be a larger part of the total time when the equations for
the primal/dual variables are simple.

\begin{figure}[h!]
  \centering
  \includegraphics[width=3.12in]{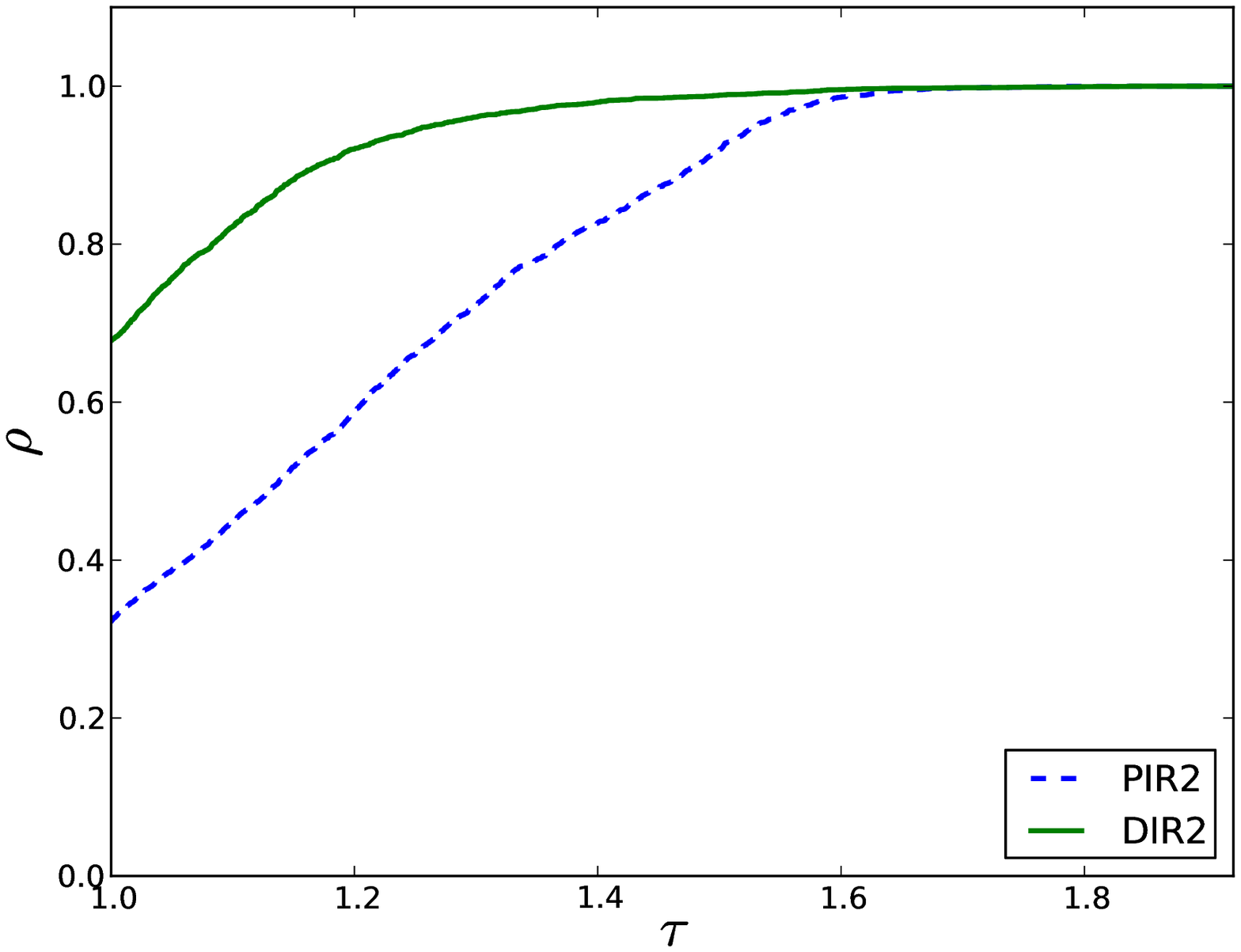}
  \includegraphics[width=3.12in]{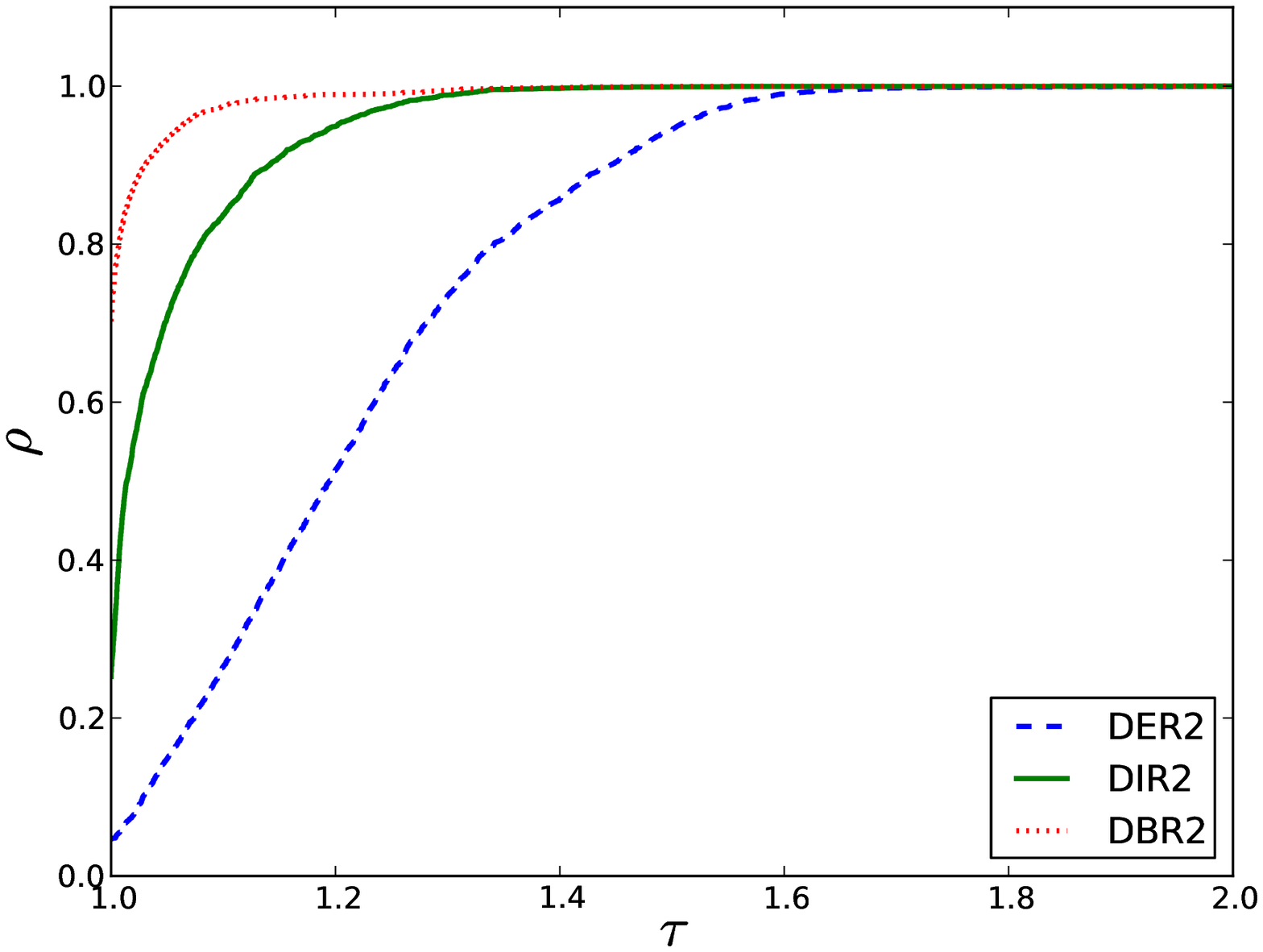}
  \caption{The leftmost figure shows the performance profiles for PIR2 and
    DIR2 and the rightmost for DER2, DIR2 and DBR2. The numerical experiment
    was done according to the description in Section \ref{sec:dop} for $n = 50
    \cdot 10^3; 100 \cdot 10^3; 200 \cdot 10^3; 500 \cdot 10^3; 10^6; 2 \cdot
    10^6$ and for the problems in Section \ref{sec:limitations} . The
    algorithms were implemented uniformly.}
\label{bhvsdualbh}
\end{figure}

We recall that DER2 applies explicit evaluation, DIR2 applies implicit
evaluation and DBR2 applies the theoretically most profitable evaluation.  The
results are in favour of DBR2 and DIR2; see the rightmost performance profile
in Figure \ref{bhvsdualbh}. DBR2 is fastest in $70.6\%$, DIR2 is fastest in
$25.9\%$ and DER2 is fastest in $4.7\%$ of the problems solved. (Note that
$70.6\% + 25.9\% + 4.7\% = 100.2 \%$; in some cases two algorithms are equally
fast.)

Figure \ref{bhvsdualbh} also shows that the performance of DIR2 and DBR2 are
quite similar. For just a few cases DIR2 is more than 1.30 times slower than
the fastest algorithm while for as few cases DBR2 is only 1.2 times slower
than the fastest algorithm. On average DBR2 performs somwhat better than DIR2.

\paragraph{Conclusion} 
For the problem set considered, it is more profitable to evaluate the dual
variable even if we have to compute all the breakpoints at the beginning of
the algorithm. Hence for the problem set considered DIR2 seems to outperform
PIR2.

For the problem set considered, the results show that in most cases it is more
profitable to evaluate a solution $\mathbf{x}^k$ implicitly.  The small
difference between the performance of DBR2 and DIR2 implies that in most cases
$|J^k|>2|L(\mu^k) \cup U(\mu^k)|$. We conclude that DBR2 performs slightly
better than DIR2 which agrees with the theory in Section \ref{sec:bh4}.

\subsection{Evaluation of the pegging process}\label{sec:resrank}

We compare the three pegging approaches for the breakpoint and relaxation
algorithms described in Sections \ref{sec:medians}--\ref{sec:MB5} and
\ref{sec:tandf}, respectively. For the breakpoint algorithm median search is
implemented similarly to \cite[Section 8.5]{PTVF92}.

The performance profiles in Figure \ref{rmrmm} show that 5-sets pegging is
nearly always the fastest (MB5 in $95\%$ and DBR5 in $94\%$ of the problem
instances solved).  Additionally, when 5-sets pegging is not fastest it is
never more than $10 \%$ slower than the fastest algorithm.  It is obvious that
for the few cases when 2-sets pegging performs best are when few optimal
primal variables equals the lower or upper bounds ($| H | /n$ is small). This
follows from the fact that 2-sets pegging does not check if the variables
belong to the additional sets that are used in 3- and 5-sets pegging.

\begin{figure}[h!]
  \centering
  \includegraphics[width=3.12in]{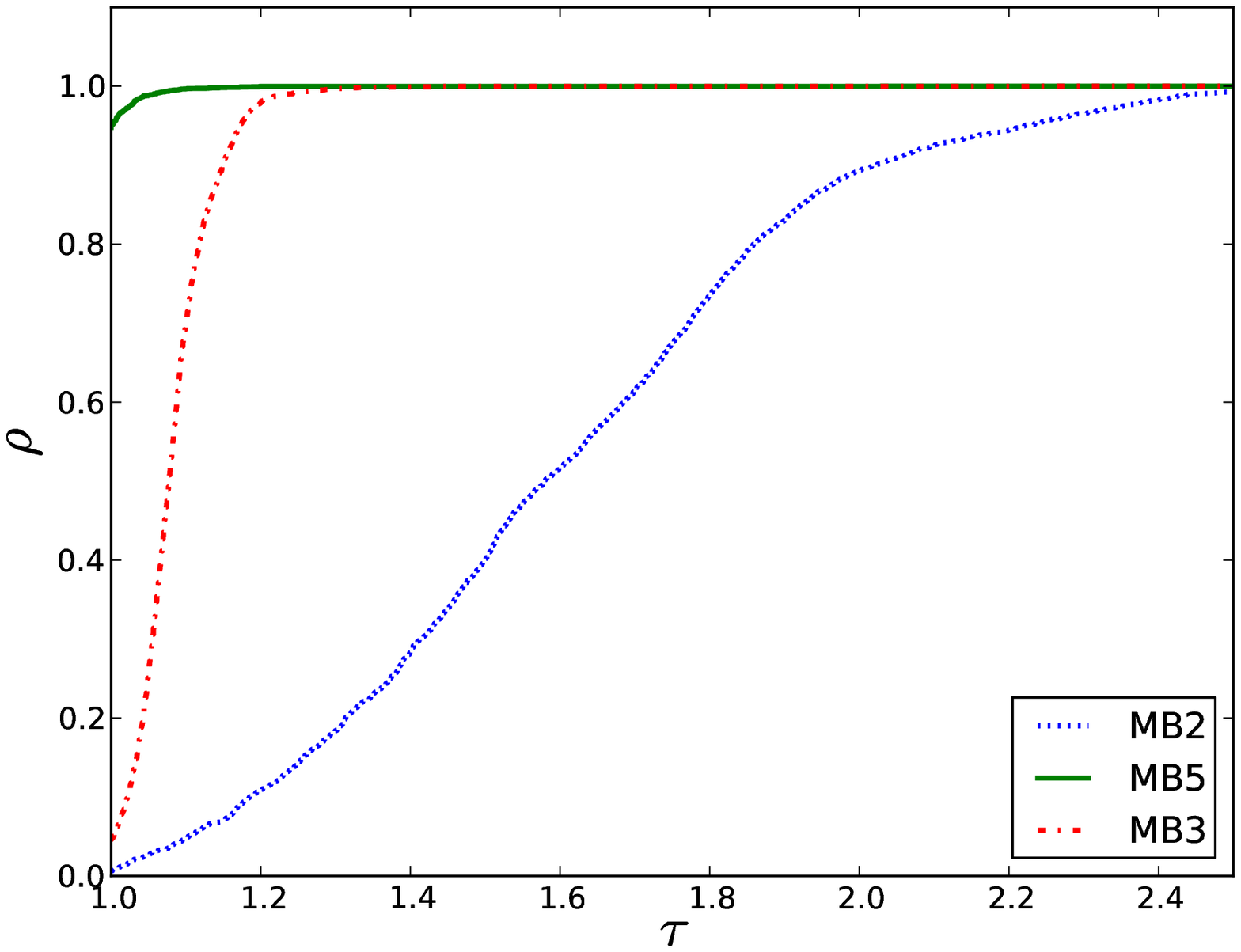}
   \includegraphics[width=3.12in]{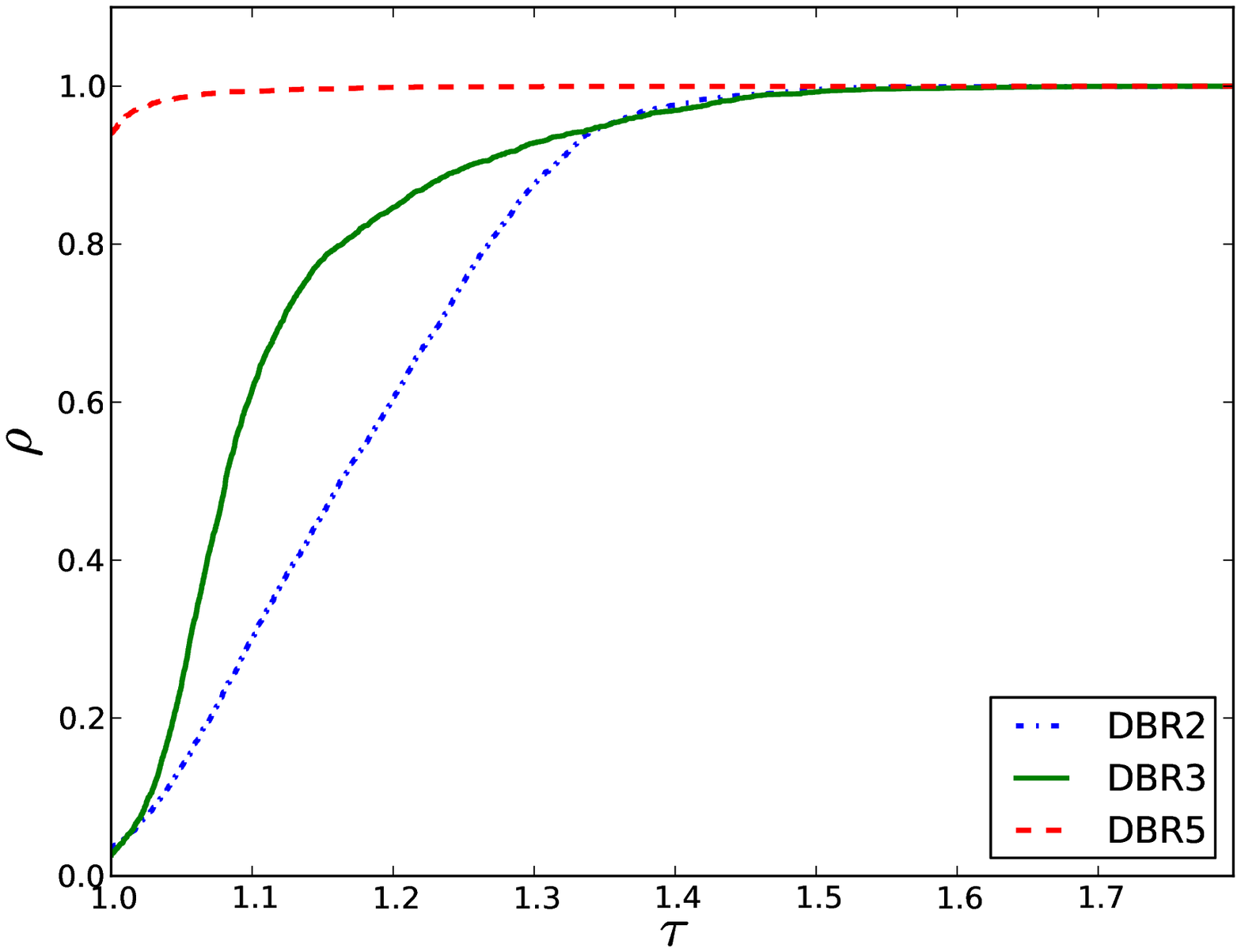}
  \caption{The leftmost figure shows the performance profile for MB2, MB3 and
    MB5 and the rightmost for DBR2, DBR3 and DBR5. The numerical experiment
    was done according to the description in Section \ref{sec:dop} for $n = 50
    \cdot 10^3; 100 \cdot 10^3; 200 \cdot 10^3; 500 \cdot 10^3; 10^6; 2 \cdot
    10^6$ and for the problems in Section \ref{sec:limitations}. The
    algorithms were implemented uniformly. }\label{rmrmm}
\end{figure}

\paragraph{Conclusion} In general 5-sets pegging is the most profitable.  

\begin{remark}
Numerical experiments for a breakpoint algorithm applying bisectional search
of a sorted sequence of breakpoints, using quicksort as in \cite[Section
  5.2.2]{Knu98}, was also performed. The algorithm was on average 1.5 times
slower than MB2. Also the algorithm in \cite{Zip80}, where the upper bounds
are relaxed, was evaluated. The algorithm performed poorly, on average 2.9
times as slow as the breakpoint algorithm applying sorting.
\end{remark}


\subsection{A comparison between relaxation, breakpoint, and quasi-Newton 
methods}\label{sec:final}

In this section we compare the best performing relaxation algorithm (DBR5),
the best performing breakpoint algorithm (MB5), and the numerical quasi-Newton
method (NZ) in Section \ref{sec:zen_algor}. For NZ we use the stopping
criterion $| \frac{\sum_{j\in J} a_jx_j}{b} - 1 | < 0.01$ which is weaker than
the stopping criterion in \cite{NiZ92} ($<10^{-4}$).  This will of course give
us an approximate optimal solution only. The initial value of the dual
variable $\mu$ is set to the mean of the breakpoints $\mu^0 = ( \sum_{j\in J}
\phi_j(l_j) / a_j + \sum_{j\in J} \phi_j(u_j) / a_j) / (2n)$. If the algorithm
does not converge within 100 CPU seconds, we start over with the mean of the
breakpoints corresponding to the lower breakpoints $\mu^0 = ( \sum_{j\in J}
\phi_j(l_j) / a_j) / n$. Similarly, if the algorithm does not converge within
100 CPU seconds then we start over with the mean of the breakpoints
corresponding to the upper breakpoints: $\mu^0 = ( \sum_{j\in J} \phi_j(u_j) /
a_j) / n$. If the algorithm does not terminate within $300$ CPU seconds then
we terminate the algorithm and consider the problem as unsolved.

The reader may have noticed that the breakpoint algorithms in Sections
\ref{sec:medians}--\ref{sec:MB5} are given without the use of the sets for the
pegging process ($L$, $U$, $M$, $L_-$, $U_+$) while the relaxation algorithms
in Section \ref{sec:bih}--\ref{sec:tandf} are given with the use of these
sets. Of course the notation using these sets, as for the relaxation
algorithms, are theoretically more elegant. However, our implementations of
the algorithms in Fortran 95 showed that for breakpoint algorithms it was in
general more profitable to not use the pegging sets.

Figure \ref{finall} shows the performance profile for the algorithms. In
general, we can see a significant advantage of using DBR5 since it is fastest
for $99.0 \%$ of the problem instances solved. Moreover, when DBR5 is not the
fastest algorithm it is almost as fast as the fastest. From the figure we also
see that MB5 is more than 2.7 times slower than the fastest algorithm in $50
\%$ of the problem instances solved.

\begin{figure}[h!]
  \centering
  \includegraphics[width=3.1in]{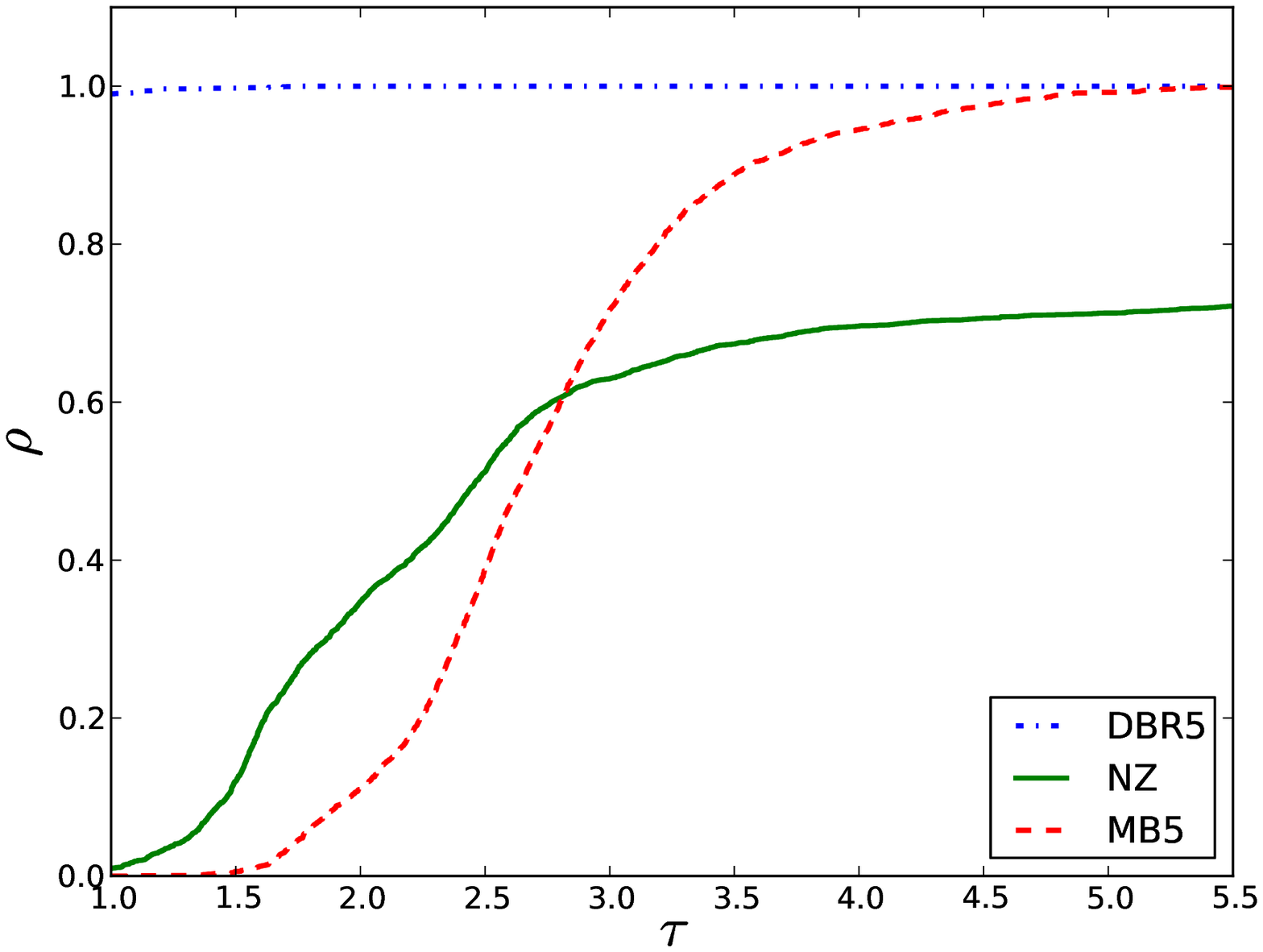}
  \includegraphics[width=3.1in]{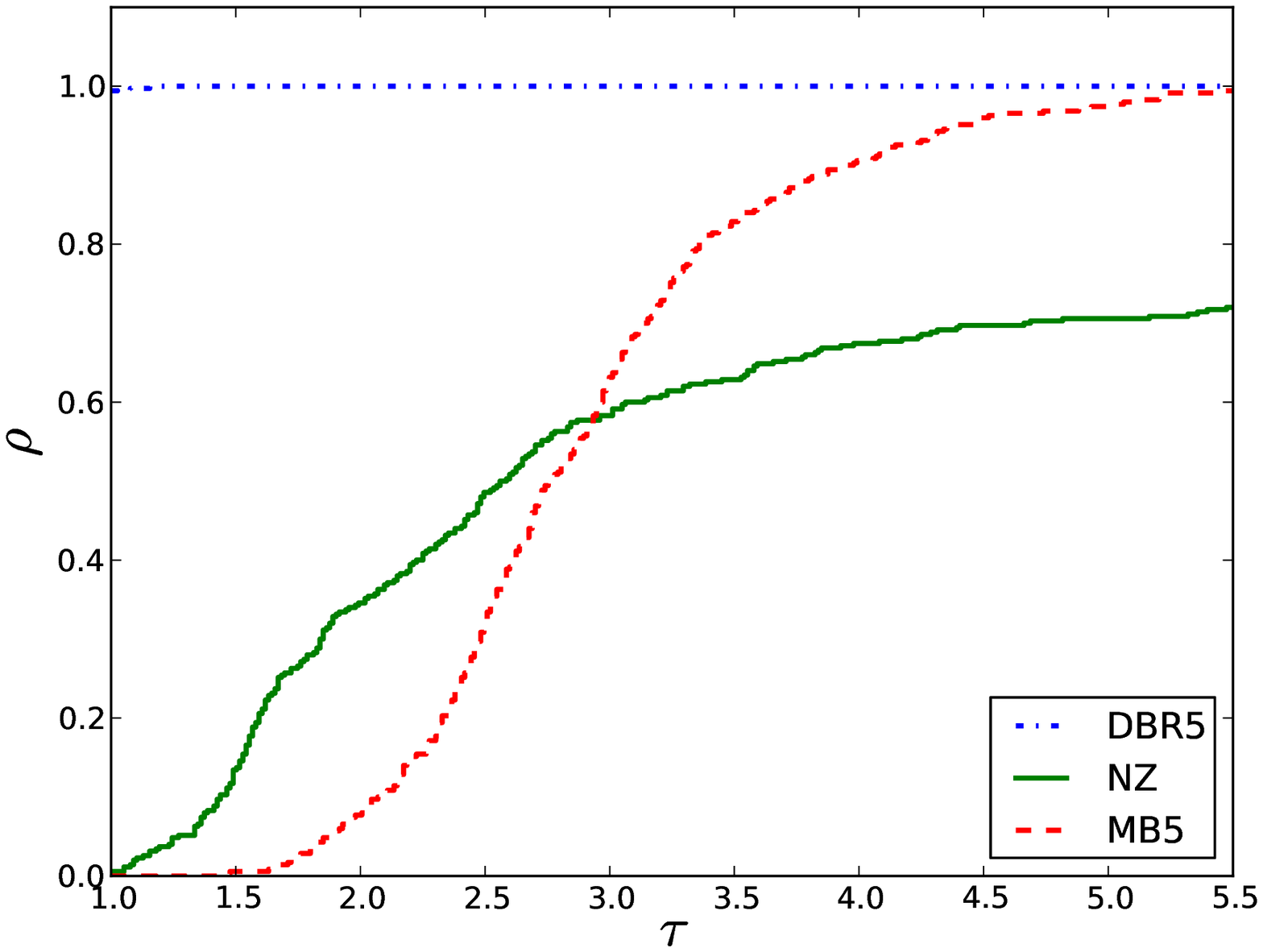}
  \caption{Performance profile for DBR5, MB5 and NZ. (Left) The numerical
    experiment was done according to the description in Section \ref{sec:dop}
    for $n = 50 \cdot 10^3$, $100 \cdot 10^3$, $200 \cdot 10^3$, $500 \cdot
    10^3$, $10^6$, and $2 \cdot 10^6$ for the problems in
    Section~\ref{sec:limitations}. (Right) The numerical experiment was done
    according to the description in Section \ref{sec:dop} but with 10 problem
    instances and for $n = 4 \cdot 10^6$, $6 \cdot 10^6$, $8 \cdot 10^6$, $10
    \cdot 10^6$, $15 \cdot 10^6$, $20 \cdot 10^6$, $25 \cdot 10^6$, and $30
    \cdot 10^6$ for the problems in
    Section~\ref{sec:limitations}. }\label{finall}
\end{figure}

The Newton method terminated the fastest for only $1 \%$ of the test problems,
which probably is due to good initial values. Also we should have in mind that
NZ terminates with an approximate solution only. The Newton method performs
relatively well for the quadratic problem (\ref{quadratic}), the theory of
search problem (\ref{TEO}), and the negative entropy problem
(\ref{negentropy}) when $|H|/n>0.8$. It is not very successful for the
stratified sampling problem (\ref{stratified}) and the sampling problem
(\ref{production}). We can see from the performance profile that in $28\%$ of
the problem instances tested, NZ is more than 5.5 times slower than DBR5 (the
fastest algorithm). In $5.3 \%$ $(158/3000)$ of the problem instances the
algorithm does not solve the problem; these problem instances mainly stem from
the stratified sampling problem (\ref{stratified}) and the sampling problem
(\ref{production}) when $|H|/n< 0.3$.

By comparing the performance profiles in Figure \ref{finall}, we can see that
they are very similar. In other words, relative to each other, the performance
of the algorithms seems to be similar for different problem sizes $n$.

\begin{figure}[h!]
  \centering \includegraphics[width=3.5in]{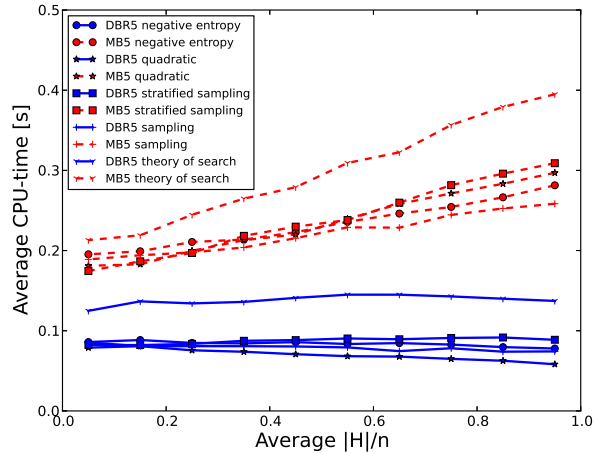}
  \caption{The average CPU time for $n=10^6$ is plotted as a function of the
    average portion of the optimal variables that obtain a value between the
    lower and upper bounds. }\label{htgraf}
\end{figure}

Considering the breakpoint algorithm MB5, Figure \ref{htgraf} shows that the
CPU time increases when the number of optimal primal variables strictly within
the bounds $|H|/n$ grows for the problems in the problem set. Concerning the
relaxation algorithm DBR5 the result shows no significant dependence of
$|H|/n$, neither does NZ.

\begin{figure}[h!]
  \centering
  \includegraphics[width=3.1in]{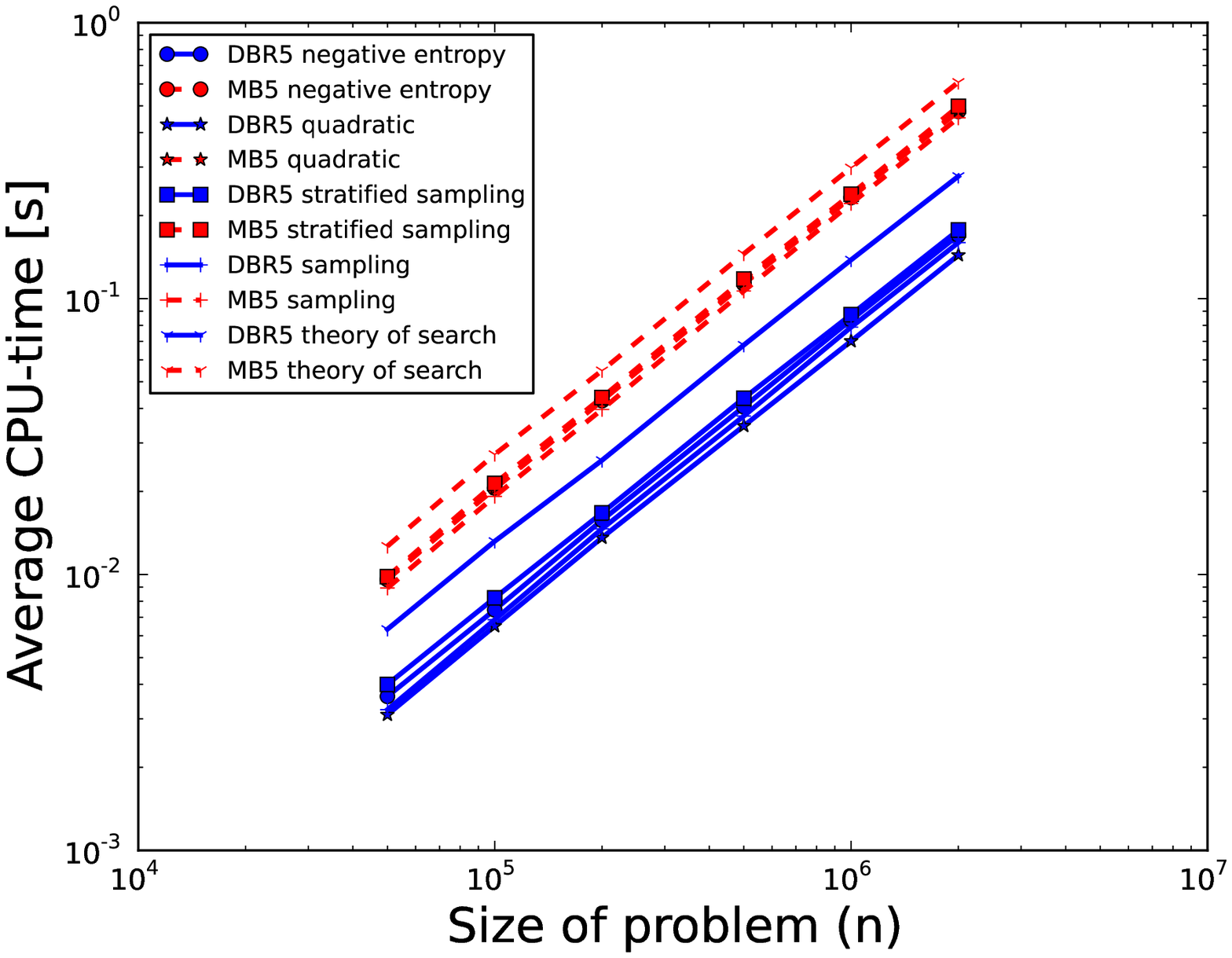}
  \includegraphics[width=3.1in]{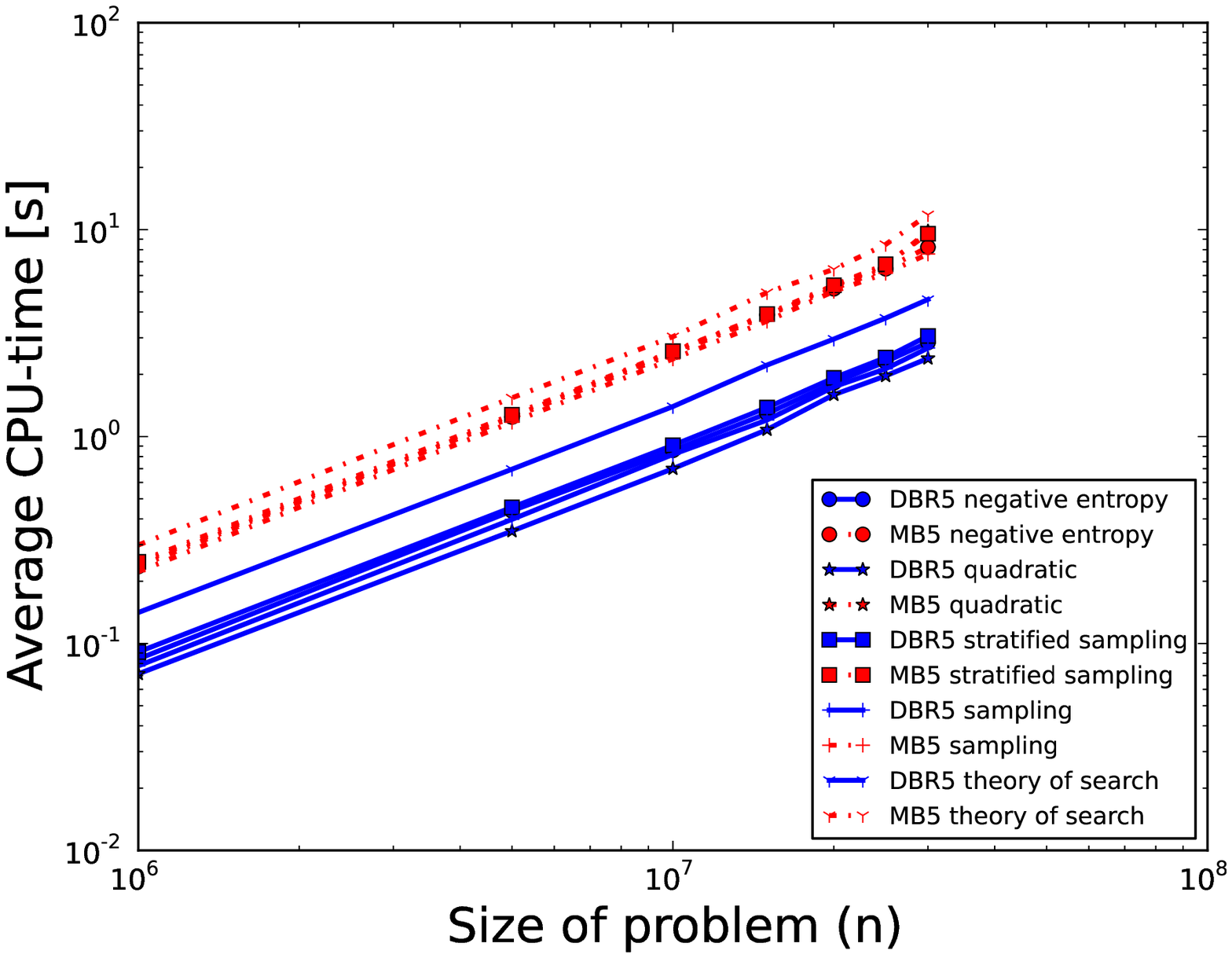}
  \caption{The average CPU time is plotted as a function of the size of the
    problem $n$ for DBR5 and MB5. The values in the left figure are the mean
    of 100 problem instances solved and the values in the right figure are the
    mean of 10 problem instances solved.}\label{naverage}
\end{figure}

In Figure \ref{naverage}, we can see that the CPU time for MB5 and DBR5 is
linear for $n \in [50 \cdot 10^3, 2 \cdot 10^6]$ respectively $n \in [50 \cdot
  10^3, 30 \cdot 10^6]$. This is impressive, since DBR5 has a worst-case time
complexity of $O(n^2)$. For problem sizes a bit larger than $30 \cdot 10^6$
variables the CPU time tends to increase faster.  However, additional
computations on a more powerful computer have shown that, for the relaxation
algorithm, the quadratic problem has a linear time complexity up to $n= 90
\cdot 10^6$, the sampling problem up to $n = 100 \cdot 10^6$, the theory of
search problem up to $n = 70 \cdot 10^6$ and the negative entropy problem up
to $n = 110 \cdot 10^6$. Hence we would also like to stress that the linearity
probably is constrained by the memory of the computer rather than the algebra
in the algorithm.

Tables \ref{tqua}--\ref{tneg} show the mean CPU time in seconds, for solving
the problems in the problem set for DBR5, NZ and MB5. The tables show the
great advantage of using DBR5.

\begin{table}[h!] 
\centering
\caption{The average CPU times algorithms MB5, DBR5 and NZ for solving the
  quadratic problem (\ref{quadratic}). Each value is the mean of 100
  randomized computations and the CPU times are given in seconds.}
\begin{tabular}{@{} c  c  c  c  c  c  c @{} } \label{tqua}
$n$ & 50,000 & 100,000 & 200,000 & 500,000 & 1,000,000 &
  2,000,000\\ [1ex] \hline DBR5 & 0.0031 & 0.0065 & 0.0136 & 0.0346 & 0.0702 &
  0.1438 \\[1ex] NZ & 0.0083 & 0.0179 & 0.0378 & 0.1018 & 0.1923 & 0.3822
  \\[1ex] MB5 & 0.0097 & 0.0205 & 0.0428 & 0.1144 & 0.2346 & 0.4807 \\[1ex]
  \hline
\end{tabular}
\end{table}

\begin{table}[h!]\
\centering
\caption{The average CPU times algorithms MB5, DBR5 and NZ for solving the
  stratified sampling problem (\ref{stratified}). Each value is the mean of
  100 randomized computations and the CPU times are given in seconds. For NZ
  only the cases when NZ found the optimal solution have been considered.}
\begin{tabular}{@{} c  c  c  c  c  c  c @{}} \label{tstr}
$n$ & 50,000 & 100,000 & 200,000 & 500,000 & 1,000,000 &
  2,000,000\\ [1ex] \hline DBR5 & 0.0040 & 0.0082 & 0.0167 & 0.0436 & 0.0875 &
  0.1773 \\[1ex] NZ & 0.0379 & 0.0952 & 0.1529 & 0.3560 & 0.9026 & 1.8299
  \\[1ex] MB5 & 0.0098 & 0.0214 & 0.0438 & 0.1178 & 0.2390 & 0.4981 \\[1ex]
  \hline
\end{tabular}
\end{table}

\begin{table}[h!]
\centering
\caption{The average CPU times algorithms MB, DBR5 and NZ for solving the
  sampling problem (\ref{production}). Each value is the mean of 100
  randomized computations and the CPU times are given in seconds.  For NZ only
  the cases when NZ found the optimal solution have been considered.}
\begin{tabular}{@{} c  c  c  c  c  c  c @{}} \label{tpro}
$n$ & 50,000 & 100,000 & 200,000 & 500,000 & 1,000,000 &
  2,000,000\\ [1ex] \hline DBR5 & 0.0032 & 0.0069 & 0.0146 & 0.0375 & 0.0788 &
  0.1597 \\[1ex] NZ & 0.1069 & 0.2403 & 0.5670 & 1.4319 & 3.8036 & 6.3949
  \\[1ex] MB5 & 0.0089 & 0.0192 & 0.0397 & 0.1066 & 0.2213 & 0.4522 \\[1ex]
  \hline
\end{tabular}
\end{table}

 \begin{table}[h!]
\centering
\caption{The average CPU times algorithms MB5, DBR5 and NZ for solving the
  theory of search problem (\ref{TEO}). Each value is the mean of 100
  randomized computations and the CPU times are given in seconds.}
\begin{tabular}{@{} c  c  c  c  c  c  c @{}} \label{ttos}
$n$ & 50,000 & 100,000 & 200,000 & 500,000 & 1,000,000 &
  2,000,000\\ [1ex] \hline DBR5 & 0.0063 & 0.0132 & 0.0259 & 0.0679 & 0.1382 &
  0.2778 \\[1ex] NZ & 0.0187 & 0.0434 & 0.0940 & 0.2669 & 0.4511 & 0.9258
  \\[1ex] MB5 & 0.0127 & 0.0273 & 0.0549 & 0.1454 & 0.2983 & 0.6090 \\[1ex]
  \hline
\end{tabular}
\end{table}

\begin{table}[h!]
\centering
\caption{The average CPU times algorithms MB5, DBR5 and NZ for solving the
  negative entropy problem (\ref{negentropy}). Each value is the mean of 100
  randomized computations and the CPU times are given in seconds.}
\begin{tabular}{@{} c  c  c  c  c  c  c @{} } \label{tneg}
$n$ & 50,000 & 100,000 & 200,000 & 500,000 & 1,000,000 & 2,000,000\\ [1ex]
  \hline DBR5 & 0.0036 & 0.0074 & 0.0157 & 0.0406 & 0.0837 & 0.1698 \\[1ex] NZ
  & 0.0071 & 0.0171 & 0.0310 & 0.0819 & 0.1709 & 0.3191 \\[1ex] MB5 & 0.0096 &
  0.0206 & 0.0427 & 0.1131 & 0.2326 & 0.4774 \\[1ex] \hline
\end{tabular}
\end{table}

\paragraph{Summary} 
Even if NZ only finds an approximate solution and even if the stop criterion
is weak, NZ is outperformed by DBR5. Hence, we see no reason to evaluate
methods for finding the exact optimal solution $\mathbf{x}^*$ from the
approximate solution $\mathbf{x}^*_\varepsilon$ generated by NZ. In many
cases, NZ fails to solve the stratified sampling problem (\ref{stratified})
and the sampling problem (\ref{production}). This is probably due to the
stiffness of the problems.

MB5 performs very well for problems where $|H|/n$ is small. This is due to the
pegging process: since we peg a large part of $x_j^k \in J^k$ in each
iteration, the problem is reduced in the next iteration. This follows from the
reduction of breakpoints by half in each iteration; note that this does not
hold for the relaxation algorithm. This might be the reason for MB5:s
dependency of $|H|/n$. However, DBR5 performs best for almost all of the
problem instances solved and it is worth noting that DBR5 never performs
poorly.

\subsection{A critical review}\label{sec:acr}

When we set up the initial values for NZ, we tried different initial values
and concluded that the one used was the most profitable on average. However,
especially if we customize initial values for each problem, there may still be
potential improvements available.

Considering the results in Section \ref{sec:Bhprim}, we consider the dual
algorithm (DIR2) to outperform the primal algorithm (PIR2) since it performs
better in general. Then we continued to evaluate several modification of the
dual algorithm while similar modifications of the primal relaxation algorithm
were not evaluated, i.e., we don't evaluate blended evaluation or 3- and 5-sets
pegging for the primal algorithm. However since the plausible modification of
the primal algorithm would take the same form as for the modification of the
dual algorithm it is assumed that a modified dual algorithms will outperform a
similar modified primal algorithm.


\section{Conclusion}\label{sec:conclusion}

We have complemented the survey in \cite{Pat08} on the resouce allocation
problem at hand, and introduced, and critically evaluated, new implementations
of breakpoint and relaxation algorithms for its solution.

The results show that our new implementations (DIR2 and DBR2) of the
relaxation algorithm outperform the earlier algorithms (PIR2 and DER2). Hence
we should evaluate the dual variable for the relaxation algorithm, i.e., DIR2
outperforms PIR2. Moreover, it is more profitable in theory, as well as in
practice, to apply blended evaluation, i.e., DBR2 outperforms DER2 and
DIR2. Our results, as well as the results in \cite{RJL92,KoL98,Kiw08b}, imply
that the relaxation algorithm is to prefer when a closed form of the dual
variable $\mu$ can be found.

We introduced 3- and 5-sets pegging for the relaxation algorithm and showed
that it is most profitable to apply 5-sets pegging, which also holds for the
breakpoint algorithm.

For the problems considered, MB5 and DBR5 have a practical time complexity of
$O(n)$ also for very large values of $n$. We also showed that the relaxation
algorithm DBR5 performs better than both the Newton-like algorithm NZ and the
breakpoint algorithm MB5. Potential future improvements include the
implementation of a pegging method for a Newton-like algorithm and/or a hybrid
of the different algorithms (NZ, DBR5 and MB5), and, hence, it would be of
interest to compare the best algorithms from our study to these.

The findings made herein can most certainly be profitably utilized also in the
efficient solution of the more complex versions of the resource allocation
problem discussed, for example, in the books \cite{Mje83,IbK88,Lus12}.


\bibliographystyle{alpha}

\bibliography{masterbib}

\end{document}